\documentclass[12pt, twoside, english, headsepline]{article}

\usepackage[margin=30mm]{geometry}
\usepackage{tikz,graphicx}
\usetikzlibrary{arrows}
\usetikzlibrary{patterns}

\tikzset{
  mid arrow/.style={postaction={decorate,decoration={
        markings,
        mark=at position .5 with {\arrow[#1]{stealth}}
      }}},
}
\usetikzlibrary{decorations.markings} % per colorare le regioni di un grafico
\tikzset
 {every pin/.style = {pin edge = {<-}}, 
  > = stealth, 
  flow/.style = 
   {decoration = {markings, mark=at position #1 with {\arrow{>}}},
    postaction = {decorate}
   },
  flow/.default = 0.5,   
  main/.style = {line width=1pt}
 }

\usepackage{amsmath,amsthm,amssymb,amsfonts, mathtools, amsbsy, dsfont}
\usepackage[utf8]{inputenx}
\usepackage{xcolor}
\usepackage{microtype}
\usepackage{hyperref}
\hypersetup{linkcolor=blue}

\newcommand*\dif{\mathop{}\!\mathrm{d}}  % simbolo differenziale

\newcommand{\floor}[1]{\left\lfloor #1 \right\rfloor}  % parte intera
\allowdisplaybreaks % fa spezzare le formule ``equation'' tra le pagine

\newtheorem{theorem}{Theorem}[section] %[theorem]
\newtheorem{corollary} {Corollary}[section] %[theorem]
\newtheorem{proposition}{Proposition}[section] %[theorem]
\newtheorem{lemma}[theorem]{Lemma}
\theoremstyle{definition} % plain
\newtheorem{definition}{Definition}[section] 
\newtheorem{example}{Example}[section]
\newtheorem{remark}{Remark}[section] %[theorem]
\numberwithin{equation}{section} % equazioni con numero interno alla sezione

\providecommand{\keywords}[1]
{
  \small	
  \textbf{\textit{Keywords: }} #1
}
\providecommand{\MSC}[1]
{
  \small	
  \textit{2020 MSC: } #1   
}

\title{Point processes of the Poisson-Skellam family}
\author{Fabrizio Cinque$^1$ and Enzo Orsingher$^2$\\
        \small Department of Statistical Sciences, Sapienza University of Rome, Italy \\
        \small $^1$fabrizio.cinque@uniroma1.it $^2$enzo.orsingher@uniroma1.it
}

\begin{document}

\maketitle

\begin{abstract}
We study a general non-homogeneous Skellam-type process with jumps of arbitrary fixed sizes. We express this process in terms of a linear combination of Poisson processes and study several properties, including the summation of independent processes of the same family, some possible decompositions (which present particularly interesting characteristics) and the limit behaviors. A compound Poisson representation and a discrete approximation are also presented. Then, we study the fractional integral of the process as well as the iterated integral of the running average. Finally, we consider some time-changed versions related to L\'{e}vy subordinators, connected to the Bernstein functions, and to the inverses of stable subordinators.
\end{abstract} \hspace{10pt}


\MSC{Primary 60G55, 60G22; Secondary 60G51}

%33C10 Bessel and Airy functions, cylinder functions
% 34A05 Explicit solutions, first integrals of ordinary differential equations
%35K25 Higher-order parabolic equations
% 60G55 Point processes
% 60G22 fractional processes, including Brownian motion
% 60G51 processes with independent increments< Levy processes

% ------  CORPO  ---------------------------------------------------------------------------------------

\section{Introduction}

The Skellam process was introduced in the short paper \cite{S1946} of 1946 as the difference of two independent Poisson processes $N_1,N_2$ with constant rates $\lambda_1,\lambda_2$ respectively. It was shown that $S = N_1-N_2$ has distribution 
\begin{equation}\label{introduzioneLeggeSkellam}
P\{S(t) = n\} = e^{-(\lambda_1+\lambda_2)t} \Biggl(\sqrt{\frac{\lambda_1}{\lambda_2}}\Bigg)^{n} I_{n}\big(2\sqrt{\lambda_1\lambda_2t}\big),\ \ \ n\in\mathbb{Z},\ t\ge0,
\end{equation}
where $I_\nu(z) = \sum_{k=0}^\infty (z/2)^{2k+\nu}/ (k! \Gamma(k+\nu + 1))$, with $\nu\in \mathbb{R}$ and $z\in\mathbb{C}$ is the modified Bessel function of order $\nu$. This process performs isolated unitary jumps of size $1$ or $-1$, extending its support to negative values as well. However, being a difference of Poisson processes it preserves several good properties (for instance stationarity and independence of the increments), allowing for detailed studies of its dynamics.

The Skellam process was subsequently extended by different authors and several new results have been recently obtained in a series of interesting papers considering a Skellam process of order $K$, that is a motion which performs jumps of size in the set $\{-K,\dots,-1,1,\dots,K\},\ K\in\mathbb{N}$, see for instance \cite{GKL2020, KK2024} and the references therein. The interest in this process and its extensions is due not only for its good stochastic properties, but also thanks to the wide spectrum of its possible applications, ranging from insurance applications to modeling the intensity difference of pixels in cameras \cite{KN2008} or the difference of the number of goals of two competing teams in a football game \cite{HJK2007} or the evolution of interacting populations \cite{CO2026}.

Another interesting extension of the Poisson process is the so-called Poisson process of order $K$, also known as general counting process, which permits us to describe arrivals up to $K$ units per instant. As far as we know this process has been introduced in 1984 in the short paper \cite{P1984} and it consists of a specific linear combination of independent homogeneous Poisson processes, i.e. $N^K = \sum_{i=1}^K i N_i$, which is strictly connected to the Skellam process of order $K,\ S^K$, since $S^K = N_1^K - N_2^K$, with $N_1^K,N_2^K$ being two independent Poisson processes of order $K$.

After the pioneering work of Laskin \cite{L2003} concerning the fractional Poisson process, several researchers studied different types of fractional point processes (see for instance \cite{BM2014, GKMR2014, PKS2011} and references therein), including non-homogeneous versions \cite{MV2019} and space-fractional versions, often related to a time-changing via Bernstein subordinators \cite{GOS2017, OT2015}, as well as state-dependent processes \cite{C2022, GOP2015, KV2019}. The quite common denomination of "time" and "space" fractionality derived from the modification of the time and of the space operator in the difference-differential equation governing the probability mass function of the processes.

Similarly, some fractional extensions of the Poisson and Skellam processes of order $K$ were introduced via time-changed Poisson processes, by using the inverse of a stable subordinator in case of a time fractionality and Bernstein subordinators in the case of the so-called space fractionality, see \cite{BS2024, DcMM2016, GKL2020, KK2024, KV2019, SMU2020}.
\\

In the present paper we consider a non-homogeneous generalized Skellam(-Poisson) process which can be useful to model situations in which the average rates of arrivals may vary with time. This generalization admits as particular cases the stochastic processes described in the previous works. In particular, we assume that $\mathcal{I}$ is a (finite) subset of $\mathbb{R}\setminus \{0\}$ and $N_i$ are independent non-homogeneous Poisson processes with rate functions $\lambda_i$ such that $\Lambda_i(t) = \int_0^t\lambda_i(s)\dif s<\infty$ for $t\ge0,\ i\in\mathcal{I}, $ and we study the linear combination $S = \sum_{i\in\mathcal{I}} iN_i$, its fractional integral and some fractional versions of the process, both in time and space sense.

In Section \ref{sezioneFamigliaSkellam} we begin by studying the generating probability function of the non-homogeneous generalized Skellam process and describing its behavior in a time interval of infinitesimal size. Then, we show the interesting form of the expected position at time $t\ge0$, the variance and the third central moment (asymmetry index), respectively given by $\sum_{i\in\mathcal{I}}i^n \Lambda_i(t) $ for $n=1,2,3$; this behavior seems not to hold for higher order moments, see Remark \ref{remarkMomentiSkellamGeneralizzato}.
We study the scaling and the summation (superposition) of generalized Skellam processes, showing that this family of stochastic processes is closed with respect to linear combinations. 

In Section \ref{sezioneDecomposizione} we discuss some thinning methods. We present the classic Bernoulli-type decomposition, yielding to independent Skellam processes with scaled rate functions, and a more general one where the resulting components are dependent Skellam processes with the same jump sizes, different from the original one.

Sections \ref{tempiPrimoPassaggioSkellamGeneralizzato} and \ref{sezioneLimiti} are respectively devoted to the first passage times (in the case of non-decreasing versions of the process) and the limit results resembling the law of large numbers and the central limit theorem. Under some conditions we also prove that the generalized Skellam process converges weakly to a Gaussian process. 
Finally, in Section \ref{sottosezioneCasoOmogeneo} we provide more details on the homogeneous case, giving a compound Poisson representation and a discrete approximation generalizing the binomial one for the homogeneous Poisson process. 

Section \ref{sezioneIntegraleSkellam} is devoted to the study of the Dzerbashyan-Caputo fractional integral of the non-homogeneous generalized Skellam process. Further results are derived for the integral of the homogeneous process, providing a compound Poisson representation and also a result concerning the iterated running average. 

Finally, Section \ref{sezioneSkellamFrazionario} concerns the fractional versions of the Skellam process. In particular, for the homogeneous case we prove a compound Poisson approximation (which turns into an exact representation for subordinators with integrable L\'{e}vy measure) and that both space and time fractionality induce a time-change of the stochastic process, respectively with a Bernstein subordinator and with the inverse of a stable subordinator (see Theorem \ref{teoremaSkellamBernsteinCaputo}). On the other hand, in the non-homogeneous case, the space fractionality does not lead to a time-changing, but we still provide some results, including the moments, the scaling and the summation.

\section{Generalized Skellam family}\label{sezioneFamigliaSkellam}

Let $(\Omega, \mathcal{F}, \{\mathcal{F}\}_t, P)$ be a filetered probaility space where all the random elements will be defined and  $\mathcal{D}[0,\infty)$ the space of the real c\`{a}dl\`{a}g functions endowed with the Skorokhod topology.

The following statement helps us to give two different mathematical definitions of the generalized Skellam process. Thus, we define this stochastic process right after Theorem \ref{teoremaDefinzioniEquivalentiSkellamGeneralizzato} (before its proof).

\begin{theorem}\label{teoremaDefinzioniEquivalentiSkellamGeneralizzato}
Let $\mathcal{I}\subset \mathbb{R}\setminus\{0\}, \ |\mathcal{I}|<\infty$, and continuous $\lambda_i : [0,\infty] \longrightarrow [0, \infty)$ such that $\Lambda_i(t) = \int_0^t \lambda_i(s)\dif s<\infty,\ \ t\ge0,\ i\in \mathcal{I}$. Let $S$ be a stochastic process such that $S(0) = 0\ a.s.$ Then, the following statements are equivalent on $\mathcal{D}[0,\infty)$:
\begin{itemize}
\item[($i$)] $S$ has independent increments and for $t\ge0,\ n\in \text{Supp}\bigl(S(t)\bigr)$,
\begin{equation}\label{definizioneInfinitesimaleSkellamGeneralizzatoNonOmogeneo}
P\{S(t+\dif t) = n+i\,|\,S(t) = n\} = 
\begin{cases}
\begin{array}{l l}
\lambda_i(t)\dif t + o(\dif t), & i\in \mathcal{I},\\
1-\sum_{j\in\mathcal{I}} \lambda_j(t) \dif t + o(\dif t), & i=0, \\
o(\dif t), & \text{otherwise}.
\end{array}
\end{cases}
\end{equation}
\item[($ii$)] Let $N_i, i\in\mathcal{I}$, be independent Poisson processes with rate functions $\lambda_i$,
\begin{equation}\label{SkellamGeneralizzatoInTerminiDiPoisson}
S(t)=\sum_{i\in\mathcal{I}} i N_i(t),\ \ \ t\ge0.
\end{equation}
\end{itemize}
\end{theorem}

\begin{definition}[Non-homogeneous generalized Skellam process]
We define a process $S$ satisfying the conditions in Theorem \ref{teoremaDefinzioniEquivalentiSkellamGeneralizzato} a \textit{non-homogeneous generalized Skellam process with rate functions $\lambda_i, i\in \mathcal{I}$}. We denote it by $S\sim NHGSP\Bigl(\mathcal{I}, (\lambda_i, i\in \mathcal{I} )\Bigr)$.
\end{definition}
In the above notation the first element represents the set of possible jumps' size and the second one the corresponding rate functions. When possible we also use the shortened notation $S\sim NHGSP(\lambda_i, i\in \mathcal{I} )$. Note that the condition $\Lambda(t)<\infty, \ t\ge0$, implies that $S$ performs a finite number of (finite) jumps in bounded intervals of time (see Remark \ref{remarkSkellamInsiemeAmpiezzaSaltiNumerabile} for the case of countable $\mathcal{I}$) and for all $t>0$ the support of the process is $\mathcal{S} = \text{Supp}\bigl(S(t)\bigr) = \big\{ \sum_{i\in \mathcal{I}} i n_i\,:\, n_i \in\mathbb{N}_0\ \forall\ i\big\}$, which is a countable set.

\begin{proof}[Proof (Theorem \ref{teoremaDefinzioniEquivalentiSkellamGeneralizzato})]

We start by observing that the process in ($ii$) has independent increments since the $N_i$ have this property for all $i\in\mathcal{I}$ and they are independent processes. Then, we readily derive the probability generating function of the process at time $t\ge0$.
\begin{equation}\label{funzioneGeneratriceSkellamGeneralizzatoNonOmogeneo}
G_{S(t)}(u) = \mathbb{E} u^{\sum_{i\in\mathcal{I}} i N_i(t)} = \prod_{i\in\mathcal{I}}\mathbb{E} \bigl(u^{i}\bigr)^{N_i(t)} = e^{-\sum_{i\in\mathcal{I}} \Lambda_i(t)(1-u^i) }.
\end{equation}

We now prove that the probability generating function of the process in ($i$) at time $t\ge0$ coincides with (\ref{funzioneGeneratriceSkellamGeneralizzatoNonOmogeneo}).
Let $p_n(t) = P\{S(t) = n\}$, from (\ref{definizioneInfinitesimaleSkellamGeneralizzatoNonOmogeneo}) we obtain the difference-differential equation
\begin{equation}\label{equazioneDifferenzeDifferenzialeSkellamGeneralizzato}
\frac{\dif }{\dif t} p_n(t)= -\sum_{i\in\mathcal{I}} \lambda_i(t) p_n(t) + \sum_{i\in\mathcal{I}}\lambda_i(t) p_{n-i}(t), \ \ \ t \ge0, n\in \mathcal{S}.
\end{equation}
Now, it holds that $ \sum_{n\in \mathcal{S}}  u^n p_{n-i}(t) = u^i \sum_{n\in \mathcal{S}-i}  u^n p_{n}(t) = u^i \sum_{n\in \mathcal{S}}  u^n p_{n}(t)= u^i G_t(u) $ for suitable $u$ in a neighborhood of $0$. Then, we can transform (\ref{equazioneDifferenzeDifferenzialeSkellamGeneralizzato}) into 
\begin{equation}\label{equazioneDifferenzialeFunzioneGeneratrice}
 \frac{\partial }{\partial t} G_t(u) = -G_t(u)\,\sum_{i\in\mathcal{I}} \lambda_i(t)(1-u^i ), \ \ \ t\ge0.
\end{equation}
It is now straightforward to see that the probability generating function emerging from (\ref{equazioneDifferenzialeFunzioneGeneratrice}) coincides with  (\ref {funzioneGeneratriceSkellamGeneralizzatoNonOmogeneo}). Finally, the independence of the increments and the exponential form of the probability generating function \eqref{funzioneGeneratriceSkellamGeneralizzatoNonOmogeneo} imply the equality in distribution of the increments and therefore the equality on $\mathcal{D}[0,\infty)$.
\end{proof}

From the probability generating function (\ref{funzioneGeneratriceSkellamGeneralizzatoNonOmogeneo}) we derive the moment generating function, with suitable $\mu\in\mathbb{R}$ (meaning that it allows convergence),
\begin{equation}\label{funzioneGeneratriceMomentiSkellamGeneralizzatoNonOmogeneo}
\mathbb{E} e^{\mu S(t)} =  e^{-\sum_{i\in\mathcal{I}} \Lambda_i(t)(1-e^{i\mu}) }.
\end{equation}

\begin{remark}\label{remarkIncrementiSkellamGeneralizzato}
The increments of the Skellam process define also a Skellam process with "delayed" rate function, meaning that, with $s\ge0$, $\{S(t+s)-S(s)\}_{t\ge0} \sim NHSKP(\lambda^s_i, i\in \mathcal{I})$ with $\lambda_i^s(t) = \lambda_i(s+t)$. Indeed, by the independence of the increments of $S$ we derive that
\begin{align}
\mathbb{E}u^{S(t+s)-S(s)}& = e^{-\sum_{i\in\mathcal{I}}\bigl(\Lambda_i(t+s)-\Lambda_i(s)\bigr)(1-u^i) }  \nonumber \\
&= e^{-\sum_{i\in\mathcal{I}}\Lambda_i(s,t+s)(1-u^i)} = e^{-\sum_{i\in\mathcal{I}}\int_0^t\lambda_i(z+s)\dif z(1-u^i)}.\label{incrementoSkellamGeneralizzato}
\end{align}

Note that if and only if the rate functions are constant in time, the increments of the process are also stationary. In this case we have a homogeneous version, further discussed below, see Section \ref{sottosezioneCasoOmogeneo}.
\hfill$\diamond$
\end{remark}

\begin{example}[Poisson process of order $K$]
If $\mathcal{I} = \{1\}$, $S$ reduces to the non-homogeneous Poisson process. Let $K\in\mathbb{N}$. If $\mathcal{I} = \{1,\dots,K\}$, $S$ is the so called Poisson process of order $K$ presented in \cite{P1984}, also known as generalized counting process (see \cite{BS2024, DcMM2016}).
\hfill$\diamond$
\end{example}

\begin{example}[Skellam process of order $K$]
If $\mathcal{I} = \{-1,1\}$, $S$ is a classical non-homogenous Skellam process. It is well-known that $S(t)$ is a Skellam random variable and its distribution is expressed in terms of the modified Bessel function (similarly to \eqref{introduzioneLeggeSkellam}). Let $S\sim NHGSP (\lambda_{1}, \lambda_{-1})$,
\begin{equation}\label{probabilitaMassaSkellamClassico}
P\{S(t) = n\}= e^{-\bigl(\Lambda_1(t)+\Lambda_{-1}(t)\bigr)} \Biggl(\sqrt{\frac{\Lambda_1(t)}{\Lambda_{-1}(t)}}\Biggr)^n I_n\Bigl(2\sqrt{\Lambda_1(t)\Lambda_{-1}(t)}\Bigr), \ \ \ n\in\mathbb{Z}, \ t\ge0.
\end{equation}
Let $K\in\mathbb{N}$. If $\mathcal{I} = \{-K,\dots,-1,1,\dots,K\}$, $S$ is called Skellam process of order $K$ and it has been recently studied in the paper \cite{GKL2020}.
\hfill$\diamond$
\end{example}

\begin{remark}[Moments]\label{remarkMomentiSkellamGeneralizzato}
From (\ref{SkellamGeneralizzatoInTerminiDiPoisson}) or the moments generating function (\ref{funzioneGeneratriceMomentiSkellamGeneralizzatoNonOmogeneo}) it is easy to derive the moments of the generalized Skellam process. For $t\ge0$,
\begin{align}
&\mathbb{E}S(t) = \sum_{i\in\mathcal{I}} i \Lambda_i(t), \ \ \  \text{Var}S(t) = \sum_{i\in\mathcal{I}} i^2 \Lambda_i(t),  \label{momentiPrimoSecondoSkellamGeneralizzatoNonOmogeneo}\\ 
& \mathbb{E} \Bigl(S(t)- \mathbb{E}S(t)\Bigr)^3 = \sum_{i\in\mathcal{I}} i^3\Lambda_{i}(t),\ \ \  \mathbb{E} \Bigl(S(t)- \mathbb{E}S(t)\Bigr)^4 = \sum_{i\in\mathcal{I}} i^4\Lambda_{i}(t) + 3\Bigl(\text{Var}S(t)\Bigr)^2. \label{momentiCentraliTerzoQuartoSekllamGeneralizzato}
\end{align}
In order to derive (\ref{momentiCentraliTerzoQuartoSekllamGeneralizzato}) we suggest to use the moment generating function, obtain the third and fourth moment of $S(t)$ respectively and then extract the central moment of interest by considering the expressions in (\ref{momentiPrimoSecondoSkellamGeneralizzatoNonOmogeneo}).
 \\
Furthermore, for $0\le s \le t$, recalling that $\text{Cov}\bigl(N_i(s), N_i(t)\bigr) = \Lambda_i(s)$,
\begin{align}
\text{Cov}\bigl((S(s),S(t)\bigr) & =\sum_{i\in\mathcal{I}} \text{Cov}\bigl(iN_i(s), iN_i(t)\bigr)  = \sum_{i\in\mathcal{I}} i^2 \Lambda_i(s) =\text{Var}S(s). \label{covarianzaSkellamGeneralizzato}
\end{align}
One can also compute the covariance without using (\ref{SkellamGeneralizzatoInTerminiDiPoisson}), but just the independence of the increments, $ \text{Cov}\bigl((S(s),S(t)\bigr) = \mathbb{E}S(s)\mathbb{E}\Bigl[S(t)-S(s)\Bigr] + \mathbb{E}S(s)^2 -  \mathbb{E}S(s)\mathbb{E}S(t) $. Note that formula (\ref{covarianzaSkellamGeneralizzato})  depends on the lower time only.

From (\ref{momentiPrimoSecondoSkellamGeneralizzatoNonOmogeneo}) follows the Fisher index $\text{FI} \bigl[ S(t) \bigr] = \frac{\text{Var}S(t)}{\mathbb{E}S(t)} = \frac{\sum_{i\in\mathcal{I}} i^2 \Lambda_i(t)}{\sum_{i\in\mathcal{I}} i \Lambda_i(t)}.$ The dispersion of $S(t)$ depends on both the rate functions and the set of the jumps size $\mathcal{I}$. We can state that if the jumps are integers or in absolute value greater or equal than $1$, $S(t)$ is over-dispersed (Fisher index $>1$). As well-known, the Skellam process is equi-dispersed (Fisher index equal to $1$) and the Skellam process of order $K$ is over-dispersed.

Finally, from the above formulas we obtain that for $0\le s<t$, with $t\rightarrow\infty$,
\begin{align}
\text{Corr}\bigl((S(s),S(t)\bigr)  =\sqrt{\frac{\sum_{i\in\mathcal{I}} i^2\Lambda_i(s)}{\sum_{i\in\mathcal{I}} i^2\Lambda_i(t)}} \sim \frac{1}{\sqrt{\sum_{i\in\mathcal{I}} i^2\Lambda_i(t)}}\sim \frac{1}{\sqrt{\Lambda_{i^*}(t)}},
\end{align}
where $\Lambda_{i^*}$ denotes the element which grows faster among the $\Lambda_i$. 
For instance, if $\lambda_i(t) \sim t^{\alpha_i}$, with $\alpha_i>-1,\ i\in\mathcal{I}$, then $\text{Cor}\bigl((S(s),S(t)\bigr)\sim t^{-(\max_{i\in\mathcal{I}} \alpha_i + 1)/2}$. We say that $S$ has long-range dependency if $\max_{i\in\mathcal{I}} \alpha_i<1$ (see Definition 3 of \cite{GKL2020} for more details on this property and particular cases of the above analysis).
\hfill$\diamond$
\end{remark}

\begin{remark}\label{remarkPoissonContaSaltiTempiDiArrivo}
In light of (\ref{SkellamGeneralizzatoInTerminiDiPoisson}) it is straightforward to derive that, for $i\in\mathcal{I}$, the Poisson process $N_i$ counts the jumps of size $i$ and the total number of jumps is given by the process $\sum_{i\in\mathcal{I}} N_i $ which is a Poisson process with rate function $\sum_{i\in\mathcal{I}} \lambda_i$.
With this at hand, it follows that the arrival times of the jumps are those of the above mentioned Poisson processes. In the case of a homogeneous generalized Skellam process, $\lambda_i(t) = \lambda_i,\ t\ge0$, the waiting times are independent exponentially distributed random variables.
\hfill$\diamond$
\end{remark}

\begin{proposition}\label{proposizioneCombinazioneLineareProcessiSkellamGenerlizzati}
\begin{itemize}
\item[($i$)] Let $S\sim NHGSP\Bigl(\mathcal{I}, (\lambda_i, i\in \mathcal{I} )\Bigr)$, then, with $a\in\mathbb{R}$
\begin{equation}\label{tesiSkellamPerCoefficiente}
aS\sim NHGSP\Bigl(a\mathcal{I}, (\lambda_{i/a}, i\in a\mathcal{I} )\Bigr).
\end{equation}
\item[($ii$)] Let $S_n\sim \Bigl(\mathcal{I}_n, (\lambda_i^{(n)},\ i\in \mathcal{I}_n )\Bigr)$, with $n\in\mathbb{N}$, be independent generalized Skellam processes such that $|\bigcup_{n=1}^\infty  \mathcal{I}_n |<\infty$ and $\sum_{n=1}^\infty \Lambda^{(n)}_i(t) <\infty,\  t\ge0, \ i\in\bigcup_{n=1}^\infty  \mathcal{I}_n$. Then,
\begin{equation}
\sum_{n=1}^\infty S_n\sim NHGSP\Biggl(\bigcup_{n=1}^\infty\mathcal{I}_n, \Bigl(\sum_{n=1}^\infty\lambda_i^{(n)},\ i\in \bigcup_{n=1}^\infty\mathcal{I}_n \Bigr)\Biggr)
\end{equation}
 where we define $\lambda_i^{(n)}, \Lambda_i^{(n)}\equiv 0 $ if  $i\not\in\mathcal{I}_n$.
\end{itemize}
\end{proposition}

Note that in point ($i$) the intensity functions remain the same and only the jumps' size changes. 

Points ($i$) and ($ii$) mean that the generalized Skellam processes define a class closed  with respect to linear combinations (under the assumptions in ($ii$), which are not required in the case of finite linear combinations). 

\begin{proof}  %($i$) We consider the probability generating function (\ref{funzioneGeneratriceSkellamGeneralizzatoNonOmogeneo}) and observe that
%\begin{equation}
%\mathbb{E} u^{aS(t)}= e^{-\sum_{i\in\mathcal{I}} \Lambda_i(t)(1-u^{ai}) } = e^{-\sum_{i\in\mathcal{aI}} \Lambda_{i/a}(t)(1-u^i) },
%\end{equation}
%which is the probability generating function of the process (\ref{tesiSkellamPerCoefficiente}) at time $t\ge0$.
By means of (\ref{SkellamGeneralizzatoInTerminiDiPoisson}) point ($i$) readily follows and, concerning point ($ii$), we have that for each $t\ge0$
\begin{equation}
\sum_{n=1}^\infty S_n(t) = \sum_{n=1}^\infty \sum_{i\in\mathcal{I}_n} i N_i^{(n)}(t) = \sum_{i\in\bigcup_n\mathcal{I}_n} i\ \sum_{n=1}^\infty N_i^{(n)}(t),
\end{equation}
where in the last term $N_i^{(n)} \equiv 0$ if $i\not\in\mathcal{I}_n$. The proof concludes by observing that under the hypotheses on the rate functions the series of Poisson processes is a Poisson process with rate function $\sum_{i\in\mathcal{I}_n} \lambda_i^{(n)}$.
\end{proof}

\begin{example}
Let $K\in\mathbb{N}$ and $S_1,\dots,S_K$ be independent classical Skellam processes with rates functions $\lambda_1^{(k)},\lambda_{-1}^{(k)},\ k=1,\dots, K$. The process $-S_1$ has rate function $\lambda_1^{(1)}$ to move one step downward and $\lambda_{-1}^{(1)}$ to move one step upward. In light of ($ii$) of Proposition \ref{proposizioneCombinazioneLineareProcessiSkellamGenerlizzati} the process $S_1+\dots + S_K$ has rate function  $\lambda_{-1}^{(1)}+\dots+\lambda_{-1}^{(K)}$ to move one step downward and $\lambda_1^{(1)}+\dots+\lambda_1^{(K)}$ to move one step upward. Thus, the probability mass function of $-S_1$ and $S_1+\dots+S_2$ follows form (\ref{probabilitaMassaSkellamClassico}). In particular, for $n\in\mathbb{Z}$, 
\begin{equation}\label{probabilitaMassaSkellamClassicoSommaK}
P\big\{S_1(t)+\dots+S_K(t) = n\big\}= e^{-\left(\Lambda^+(t)+\Lambda^-(t)\right)} \sqrt{\frac{\Lambda^+(t)}{\Lambda^-(t)}} I_n\Bigl(2\sqrt{\Lambda^+(t)\Lambda^-(t)}\Bigr), \ \ \ n\in\mathbb{Z}, \ t\ge0,
\end{equation}
where $\Lambda^+ = \Lambda_{1}^{(1)}+\dots+\Lambda_{1}^{(K)}$ and $\Lambda^- = \Lambda_{-1}^{(1)}+\dots+\Lambda_{-1}^{(K)}$.
\\
Another example of Proposition \ref{proposizioneCombinazioneLineareProcessiSkellamGenerlizzati} is the well-known result that the sum and difference of Skellam processes of order $K$ are still the same type of process with suitable rate functions. We point out that the distribution of this process has been recently erroneously written like (\ref{probabilitaMassaSkellamClassicoSommaK}). From (\ref{SkellamGeneralizzatoInTerminiDiPoisson}) it is clear that the Skellam process of order $K$ is different in distribution from the sum of $K$ classical Skellam processes; indeed, if $K=2$, by using the above notation and by denoting with $S^{(2)}$ a Skellam process of order $2$, we have $S^{(2)} = \sum_{\substack{i=-2\\i\not=0}}^2 i N_i = S_1 + 2S_2\not= S_1+S_2 $.
\hfill$\diamond$
\end{example}

\begin{remark}[Countable $\mathcal{I}$]\label{remarkSkellamInsiemeAmpiezzaSaltiNumerabile}
We point out that in the case of $\mathcal{I}$ with a countable number of elements, the above results hold if we add some hypotheses which allow convergence and exchangeability of integrals and series.

From Theorem \ref{teoremaDefinzioniEquivalentiSkellamGeneralizzato} it is clear that we need a condition on the rate functions and for the convergence of the probability generating function. %Indeed, $\mathcal{I}$ can contain both positive and negative elements, creating problems when it is unbounded. One can assume % the following hypotheses: for all $t\ge0$ there exists $\varepsilon>0$ such that
%\begin{equation}\label{condizioneConvergenzaFunzioneGeneratriceProbabilita}
%\forall \ t\ge0\ \ \exists\ \varepsilon>0 \ \ s.t.\ \ \bigg| \sum_{i\in\mathcal{I}} (u^i-1)\Lambda_i(t)\bigg| <\infty,\ \ \ \varepsilon<|u|<1-\varepsilon.
%\end{equation}
%Condition (\ref{condizioneConvergenzaFunzioneGeneratriceProbabilita}) implies
Let us consider the process in a finite time interval $[0,T]$, with $T>0$. One can assume that %for all $t\ge0$, $\sum_{i\in\mathcal{I}} \Lambda_i(t) <\infty$ (which implies the hypothesis at the beginning of the statement of Theorem \ref{teoremaDefinzioniEquivalentiSkellamGeneralizzato}) and
there exists $\varepsilon\in (0,1)$ such that
\begin{equation}\label{condizioneConvergenzaFunzioneGeneratriceProbabilita}
\bigg| \sum_{i\in\mathcal{I}} u^i\Lambda_i(T)\bigg| <\infty,\ \ \ \varepsilon<|u|\le 1.
\end{equation}
Condition \eqref{condizioneConvergenzaFunzioneGeneratriceProbabilita} implies that $\sum_{i\in\mathcal{I}} \Lambda_i(t)<\infty, 0\le t \le T,$ since the $\Lambda_is$ are  non-decreasing.
Furthermore, in the proof of Theorem \ref{teoremaDefinzioniEquivalentiSkellamGeneralizzato} one needs the exchangeability between integral and series while solving equation (\ref{equazioneDifferenzialeFunzioneGeneratrice}). 

Note that the moments in Remark \ref{remarkMomentiSkellamGeneralizzato} may diverge.

We point out that if one assumes that $\mathcal{I}$ is bounded (and countable), then  $u^i$ is bounded as well and it is sufficient to assume that $\sum_{i\in\mathcal{I}} \Lambda_i(T) <\infty$. This reduces to $\sum_{i\in\mathcal{I}} \lambda_i<\infty $ in the homogeneous case (i.e. $\lambda_i(t) = \lambda_i,\ t\ge0,\ i\in\mathcal{I}$).
\hfill$\diamond$
\end{remark}

\subsection{Decomposition of Skellam processes}\label{sezioneDecomposizione}

We now discuss a quite general decomposition of a Skellam process into two subprocesses, $S_1,S_2,$ by splitting the jumps that the process performs. We refer to \cite{DK2024} for some recent results on the decomposition of a general counting process (i.e. of a Poisson-type process), which can be seen as a particular case of the below dissertation.

We consider a splitting rule in which each jump (of size $i\in\mathcal{I}$) is split among the two subprocesses $S_1$ and $S_2$ according to a fixed probabilistic rule, maintaining the sign of the jump. Formally we describe the problem as follows.

For the sake of clarity we here assume that $\mathcal{I}\subset\mathbb{Z}\setminus\{0\}$. Let $N$ be the Poisson process counting the jumps of the generalized Skellam process $S$ (see Remark \ref{remarkPoissonContaSaltiTempiDiArrivo}). For $t\ge0$, each jump $X_k$, with $k = 1,\dots, N(t)$, is split into 
\begin{equation}\label{saltiDecomposizioneSkellam}
	X^1_k = \text{sgn}(X_k)Y_k\ \text{ and } \ X^2_k = \text{sgn}(X_k) \bigl(|X_k| - Y_k\bigr)
\end{equation}
where $Y_k$ is a random variable having support contained in $\bigl[0, |X_k|\bigr]$. Then, the portion $X_k^1$ contributes to the subprocess $S_1$ and $X_k^2$ contributes to $S_2$, i.e. $S_j(t) = \sum_{k=0}^{N(t)} X_k^j,\ t\ge0, j=1,2$. 

We assume that $Y_1,Y_2,\dots$ are i.i.d. with support in $\Bigl[0, \max\big\{\max\{\mathcal{I}\}, -\min\{\mathcal{I}\}\big\}\Bigr]$. It is obvious that if $Y_k$ are continuous random variables, the split processes do not belong to the Skellam family. On the other hand, the next theorem states that if $Y_k$ are discrete, the subprocesses belong to the Skellam family.

Hereafter we denote with $q(j;i) =  P\{Y_k = j \,|\,X_k = i\},\ i\in\mathcal{I}, \ j=0,\dots, |i|, \ k\in\mathbb{N}$, also meaning that $\sum_{j=0}^{|i|} q(j;i)=1$.

\begin{theorem}\label{teoremaDecomposizioneSkellam}
	Let $\mathcal{I}\subset\mathbb{Z}$. $S\sim NHGSP (\lambda_i, i\in \mathcal{I} )$ and $S_1$ and $S_2$ be the splitted processes described above. Then, $S_1$ and $S_2$ are Skellam processes performing jumps of size in $\Big\{\min\big\{0,\min\mathcal{I}\big\},\dots, \max\big\{0,\max\mathcal{I}\big\}\Big\}\setminus \{0\}$,
	\begin{align}
		&S_1\sim NHGSP \Biggl(  \sum_{\substack{i\ge j\\i\in\mathcal{I}^+}}\lambda_i q\big(j;i\big),\ 1\le j\le \max\mathcal{I};\ \sum_{\substack{i\le j\\i\in\mathcal{I}^-}}\lambda_i q\big(|j|;i\big), \ -1\ge j\ge \min\mathcal{I} \Biggr) \label{decomposizioneSkellamPrimaComponente} \\
		&S_2\sim NHGSP \Biggl(  \sum_{\substack{i\ge j\\i\in\mathcal{I}^+}}\lambda_i q\big(i-j;i\big),\  1\le j\le \max\mathcal{I};\ \sum_{\substack{i\le j\\i\in\mathcal{I}^-}}\lambda_i q\big(|i|-|j|;i\big),\ -1\ge j\ge \min\mathcal{I} \Biggr) \label{decomposizioneSkellamSecondaComponente}
	\end{align}
	where $\mathcal{I}^+ = \mathcal{I}\cap(0,\infty)$ and $\mathcal{I}^- = \mathcal{I}\cap(-\infty,0)$.
	Furthermore, 
	\begin{equation}\label{covarianzaComponentiScomposizione}
		\text{Cov}\Big(S_1(t), S_2(t)\Big) = \sum_{i\in\mathcal{I}} \Lambda_i(t) \mathbb{E}\Bigl[Y\bigl(|X|-Y\big)\big|X=i\Bigr],\ \ \ t\ge0,
	\end{equation}
	where $X,Y$ are respectively distributed as the $X_k$ and the $Y_k$ in (\ref{saltiDecomposizioneSkellam}).
\end{theorem}

Note that in \eqref{decomposizioneSkellamPrimaComponente} and \eqref{decomposizioneSkellamSecondaComponente} we are using the short notation to describe a non-homogeneous generalized Skellam process (see after Definition \ref{definizioneInfinitesimaleSkellamGeneralizzatoNonOmogeneo}), meaning that, for instance $S_1$ permforms a jump of size $ 1\le j\le \max\mathcal{I}$ with rate function $\sum_{\substack{i\ge j\\i\in\mathcal{I}^+}}\lambda_i q\big(j;i\big)$.
We point out that in Theorem \ref{teoremaDecomposizioneSkellam}, if $\mathcal{I}\subset (0,\infty)$, $S_1$ and $S_2$ perform jumps of size in $\{1,\dots,\max\mathcal{I}\}$ and in (\ref{decomposizioneSkellamPrimaComponente}) and (\ref{decomposizioneSkellamSecondaComponente}) the part about negative $j$ does not appear.

Note that the covariance (\ref{covarianzaComponentiScomposizione}) between the two components $S_1,S_2$ is always positive.

\begin{proof}
	By means of definition (\ref{definizioneInfinitesimaleSkellamGeneralizzatoNonOmogeneo}), (\ref{saltiDecomposizioneSkellam}) and the description after the latter formula, we derive the following infinitesimal behavior of the joint distribution of $S_1$ and $S_2$. For the sake of completeness we consider the case where $\mathcal{I}\cap (-\infty,0)\not=\emptyset$, i.e. $S$ can perform negative jumps. For $n_1,n_2\in\mathbb{Z}$ and $t\ge0$, \begin{align}\label{definizioneInfinitesimaleSkellamGeneralizzatoNonOmogeneoScomposizione}
		&P\{S_1(t+\dif t) = n_1+i_1,\,S_2(t+\dif t) = n_2+i_2\,|\,S_1(t) = n_1, S_2(t) = n_2\}  \\
		& = P\{S_1(t+\dif t) = n_1+i_1,\,S_2(t+\dif t) = n_2+i_2,\, S(t+\dif t) = n_1+n_2+i_1+i_2\, \Big | \nonumber\\
		&\ \ \ \Big| \,S_1(t) = n_1,\,S_2(t) = n_2,\, S(t) = n_1+n_2\}\nonumber\\
		& = P\{S(t+\dif t) = n_1+n_2+i_1+i_2 \,|\,S(t) = n_1+n_2\}\label{passaggioDecomposizioneGenerale}\\
		&\ \ \ \times P\{S_1(t+\dif t) = n_1+i_1,\,S_2(t+\dif t) = n_2+i_2\,|\,S_1(t) = n_1,S_2(t) = n_2,S[t,t+\dif t) = i_1+i_2\}\nonumber\\
		&= \begin{cases}
			\begin{array}{l l}
				\lambda_{i_1+i_2}(t)q(i_1;i_1+i_2)\dif t + o(\dif t), & i_1+i_2\in \mathcal{I}^+,\ 0\le i_1,i_2\le \max\mathcal{I}^+,\\
				\lambda_{i_1+i_2}(t)q\big(|i_1|;i_1+i_2\big)\dif t + o(\dif t), & i_1+i_2\in \mathcal{I}^-,\ 0\ge i_1,i_2\ge \max\mathcal{I}^-,\\
				1-\sum_{i\in\mathcal{I}} \lambda_i(t)\dif t + o(\dif t), & i_1=0, i_2=0,\\
				o(\dif t), & \text{otherwise},
			\end{array}
		\end{cases}\nonumber
	\end{align}
	where in step \eqref{passaggioDecomposizioneGenerale} we used that $S(t+\dif t)$, conditionally on $S(t)$ is independent on $S_1(t)$ and $S_2(t)$.
	Note that the expression in the case of $i_1=i_2 = 0$ derives by means of the following computation (similarly applied also for $\mathcal{I}^-$),
	\begin{align*}
		\sum_{\substack{i_1,i_2=0\\i_1+i_2\in\mathcal{I}^+}}^{\max\mathcal{I}^+}\lambda_{i_1+i_2}(t)q(i_1;i_1+i_2) = \sum_{i\in\mathcal{I}^+}\sum_{j=0}^i \lambda_i(t) q(j;i) = \sum_{i\in\mathcal{I}^+}\lambda_i(t).
	\end{align*}
	Now, from (\ref{definizioneInfinitesimaleSkellamGeneralizzatoNonOmogeneoScomposizione}) we derive a difference-differential equation for the joint probability mass function $p_t(m,n) = P\{S_1(t) = m,S_2(t) = n\}$, with $m,n\in\mathbb{Z}$,
	\begin{align}
		\frac{\dif p_t(m,n)}{\dif t} = -\sum_{i\in \mathcal{I} }\lambda_i(t) p_t(m,n) &+ \sum_{i\in \mathcal{I}^+} \lambda_i(t)\sum_{j=0}^i q(j;i)p_t\big(m-j,n-(i-j)\big)\\
		& +  \sum_{i\in \mathcal{I}^-} \lambda_i(t)\sum_{j=i}^0 q\big(|j|;i\big)p_t\big(m-j,n-(i-j)\big). \nonumber
	\end{align}
	Now, by considering the joint probability generating function with suitable $u,v$ in the neighborhood of $0$, observing that, for $h,k\in\mathbb{Z}$, $\sum_{m,n=-\infty}^\infty u^m v^np_t(m-h, n-k) =u^h v^k\sum_{m,n=-\infty}^\infty u^{m-h} $ $v^{n-k} p_t(m-h, n-k) = u^h v^k G_t(u,v)$, we obtain the differential equation governing $G_t(u,v)$,
	\begin{align}
		\frac{\partial G_t(u,v)}{\partial t}  = - \Biggl(\sum_{i\in\mathcal{I}} \lambda_i(t) - \sum_{i\in \mathcal{I}^+} \lambda_i(t)\sum_{j=0}^i q(j;i) u^j v^{i-j} -  \sum_{i\in \mathcal{I}^-} \lambda_i(t)\sum_{j=i}^0 q\big(|j|;i\big) u^j v^{i-j} \Biggr)G_t(u,v).\nonumber
	\end{align}
	Hence, we have
	\begin{align}
		G_t(u,v) &= \text{exp}\Biggl(-\sum_{i\in\mathcal{I}} \Lambda_i(t) + \sum_{i\in \mathcal{I}^+} \Lambda_i(t)\sum_{j=0}^i q(j;i) u^j v^{i-j} + \sum_{i\in \mathcal{I}^-} \Lambda_i(t)\sum_{j=i}^0 q\big(|j|;i\big) u^j v^{i-j} \Biggr)\nonumber\\
		& = \text{exp}\Biggl(- \sum_{i\in\mathcal{I}}\Lambda_i(t)\biggl(1- \mathbb{E}\Bigl[u^{\text{sgn}(X)Y}v^{X-\text{sgn}(X)Y}\big|X=i\Bigr]\biggr)\Biggr),\label{funzioneGeneratriceDoppiaScomposizione}
	\end{align}
	where $X,Y$ are copies of the jump sizes in (\ref{saltiDecomposizioneSkellam}), meaning that the conditional law of $Y$ given $X=i$ is $q(\cdot;i)$.
	
	From (\ref{funzioneGeneratriceDoppiaScomposizione}) we derive the distribution of the marginal processes and the covariance structure.
	By setting $v=1$ we have
	\begin{align}
		\mathbb{E}u^{S_1(t)} &= G_t(u,1) =  \text{exp}\Biggl(- \sum_{i\in\mathcal{I}}\Lambda_i(t)\biggl(1- \mathbb{E}\Bigl[u^{\text{sgn}(X)Y}\big|X=i\Bigr]\biggr)\Biggr)\label{funzioneGeneratriceMarginaleScomposizione}\\
		& =  \text{exp}\Biggl(- \sum_{i\in\mathcal{I}}\Lambda_i(t)\Bigl(1 - q\big(0 ;i\big)\Bigr) + \sum_{j=1}^{\max\mathcal{I}^+} u^j \sum_{\substack{i\ge j\\ i\in\mathcal{I}^+}} \Lambda_i(t)q(j;i) \nonumber \\
		&\ \ \ + \sum_{j=\min\mathcal{I}^-}^{-1} u^j \sum_{\substack{i\le -j\\ i\in\mathcal{I}^-}} \Lambda_i(t)q\big(|j|;i\big) \Biggr),\nonumber
	\end{align}
	which yields the form of the process $S_1$ given in (\ref{decomposizioneSkellamPrimaComponente}). Similarly one obtains (\ref{decomposizioneSkellamSecondaComponente}).
	
	From (\ref{funzioneGeneratriceMarginaleScomposizione}) by deriving once and setting $u=1$, we obtain the moment of $S_1$, in particular, we have the following form in terms of the conditional distribution of the split jump $Y$ (and equivalently for the second component $S_2$),
	\begin{equation}\label{momentiComponentiScomposizione}
		\mathbb{E}S_1(t) = \sum_{i\in\mathcal{I}} \Lambda_i(t)\,\text{sgn}(i) \mathbb{E}\Bigl[Y|X=i\Bigr]\ \ \ \text{ and }\ \ \  \mathbb{E}S_2(t) = \sum_{i\in\mathcal{I}} \Lambda_i(t)\,\text{sgn}(i) \mathbb{E}\Bigl[|X|-Y|X=i\Bigr].
	\end{equation}
	Finally, by means of the joint probability generating function (\ref{funzioneGeneratriceDoppiaScomposizione}) we have that 
	\begin{align*}
		&\mathbb{E}S_1(t)S_2(t) = \frac{\partial G_t(u,v)}{\partial u\partial v}\Bigg|_{u=v=1} \\
		&= \sum_{i\in\mathcal{I}} \Lambda_i(t)\,\text{sgn}(i) \mathbb{E}\Bigl[Y|X=i\Bigr]\sum_{i\in\mathcal{I}} \Lambda_i(t) \,\text{sgn}(i)\mathbb{E}\Bigl[|X|-Y\big|X=i\Bigr] + \sum_{i\in\mathcal{I}} \Lambda_i(t) \mathbb{E}\Bigl[Y\bigl(|X|-Y\big)\big|X=i\Bigr],
	\end{align*}
	and by means of (\ref{momentiComponentiScomposizione}) we obtain the covariance (\ref{covarianzaComponentiScomposizione}).
\end{proof}

\begin{example}[Binomial decomposition]
	Let us assume a binomial split of the jumps, i.e. the variables $Y_k$ in (\ref{saltiDecomposizioneSkellam}) are such that $q(j;i) =  P\{Y_k = j \,|\,X_k = i\} = \binom{|i|}{j}p^j(1-p)^{|i|-j},\ i\in\mathcal{I}, \ j=0,\dots, |i|, \ k\in\mathbb{N}$. Then, the joint probability generating function (\ref{funzioneGeneratriceDoppiaScomposizione}) has the following interesting formula
	\begin{equation*}
		G_t(u,v) = \text{exp}\Biggl(- \sum_{i\in\mathcal{I}}\Lambda_i(t)\biggl(1- \Bigl[pu^{\text{sgn}(i)}+(1-p)v^{\text{sgn}(i)}\Bigr]^{|i|}\,\biggr)\Biggr),
	\end{equation*}
	from which we easily derive the marginal ones. The covariance (\ref{covarianzaComponentiScomposizione}) reads $ \text{Cov}\Big(S_1(t), S_2(t)\Big) = \sum_{i\in\mathcal{I}} \Lambda_i(t) |i|p\big(|i|-1\big) .$
	\hfill$\diamond$
\end{example}

The interested reader can notice that the decomposition discussed in Theorem \ref{teoremaDecomposizioneSkellam} can be extended to the case of $H$ subprocesses, by assuming a splitting rule based on $H$ components instead of two. This can be formalized by suitably adapting the jumps definition in (\ref{saltiDecomposizioneSkellam}); for instance by assuming $Y_k^{(1)},\dots,Y_k^{(H)}$ such that $Y_k^{(1)}+\dots+Y_k^{(H)} = |X_k|$.
\\

As a particular case of Theorem \ref{teoremaDecomposizioneSkellam} we obtain the typical Bernoulli-type decomposition, i.e. where each jump of size $i\in\mathcal{I}$ is assigned to $S_1$ with probability $p_i\in(0,1)$ and to $S_2$ with probability $1-p_i$. In this case, given the $k$-th jump $X_k=i$, the random contribution to the first subprocess, the random split $Y_k$ in \eqref{saltiDecomposizioneSkellam}, is $Y_k =|i|$ with probability $p_i\in (0,1)$ and $Y_k = 0$ with probability $1-p_i$, i.e. the probabilities in \eqref{decomposizioneSkellamPrimaComponente} and \eqref{decomposizioneSkellamSecondaComponente} reduces to $q(|i|;i) = p_i$ and $q(0;i) = 1-p_i, i\in\mathcal{I}$.

\begin{corollary}[Bernoulli decomposition]\label{proposizioneCoppiaSkellamScompostoBernoulli}
Let $S\sim NHGSP (\lambda_i, i\in \mathcal{I} )$ and $S_1$ and $S_2$ be obtained as described above. Then, $S_1$ and $S_2$ are independent processes such that
\begin{equation}\label{tesiCoppiaSkellamScompostoBernoulli}
S_1\sim NHGSP \Bigl(p_i\lambda_i, i\in \mathcal{I} \Bigr)\ \ \ \text{and}\ \ \ S_2\sim NHGSP\Bigl((1-p_i)\lambda_i, i\in \mathcal{I} \Bigr).
\end{equation}
\end{corollary}

\begin{proof}
It sufficies to observe that from \eqref{funzioneGeneratriceDoppiaScomposizione} we obtain the following joint probability generating function, for suitable $u,v$ in a neighborhood of $0$,
\begin{align}
G_{S_1(t),S_2(t)}(u,v)& = e^{-\sum_{i\in\mathcal{I}} p_i\Lambda_i(t)(1-u^i) }e^{-\sum_{i\in\mathcal{I}} (1-p_i)\Lambda_i(t)(1-v^i) }. \label{funzioneGeneratriceCoppiaSkellamScompostoBernoulli}
\end{align}
By keeping in mind (\ref{funzioneGeneratriceSkellamGeneralizzatoNonOmogeneo}), formula \eqref{funzioneGeneratriceCoppiaSkellamScompostoBernoulli} yields the claimed result.
\end{proof}

We point out that one could also derive Corollary \ref{proposizioneCoppiaSkellamScompostoBernoulli} by means of representation ($ii$) of Theorem \ref{teoremaDefinzioniEquivalentiSkellamGeneralizzato} and by using known results for the Poisson process.

Note that to extend Corollary \ref{proposizioneCoppiaSkellamScompostoBernoulli} in order to decompose a generalized Skellam process into $H$ independent generalized Skellam processes, it sufficies to consider splits governed by multinomial random variables which select the subprocess recording the jump. Assuming that a jump of size $i$ is assigned to the $h$-th subprocess with probability $p_i^{(h)},\ \sum_{h=1}^H p_i^{(h)} = 1\ \forall\ i$, then $S$ will be decomposed into the independent processes $S_h\sim NHGSP \Bigl(p_h\lambda_i, i\in \mathcal{I} \Bigr)$ for $h=1,\dots, H$.

\subsection{First passage times}\label{tempiPrimoPassaggioSkellamGeneralizzato}

Here we discuss the case of a non-decreasing Skellam process with natural jump size, that is the case when $\mathcal{I}\subset \mathbb{N}$; in this case it would be more precise to talk about a generalized Poisson process (counting process).

Let $T_n = \inf\{t\ge0\,:\,S(t)\ge n\}$, with $n\in\mathbb{N}$. We note that the process reaches at least the state $1$with its first step, therefore $T_1$ coincides with the arrival time for a non-homogeneous Poisson process with rate function $\sum_{i\in\mathcal{I}}\lambda_i$; in the case of homogeneous process this reduces to an exponential random variable. 

Now, we derive the generating function for the probabilities of the kind $q_n(t) = P\{T_n > t\}$, with $t\ge0$. First, we recall the following general relationship (which is holding for all non-decreasing counting processes over $\mathbb{N}$), for suitable $|u|< 1$ (see Appendix \ref{appendiceDimostrazioneFunzioneGeneratriceTempiPrimoPassaggio} for the details),
\begin{equation}\label{generatriceTempiPrimoPassaggioRipartizione}
	\sum_{n=1}^\infty u^n P\{T_n \le t \} = \frac{u}{u-1}\Bigl(G_t(u) - 1\Bigr).
\end{equation}
Now, by means of simple calculation and \eqref{funzioneGeneratriceMomentiSkellamGeneralizzatoNonOmogeneo} yields
\begin{equation}\label{generatriceTempiPrimoPassaggio}
Q_t(u) = \sum_{n=1}^\infty u^n P\{T_n > t \} = \frac{u}{1-u}G_t(u)= \frac{u}{1-u} e^{-\sum_{i\in\mathcal{I}} \Lambda_i(t) (1-u^i)}.
\end{equation}
By keeping in mind equation (\ref{generatriceTempiPrimoPassaggio}) one can prove that the survival distribution functions $q_n(t)$ satisfy a difference-differential equation equivalent to (\ref{equazioneDifferenzeDifferenzialeSkellamGeneralizzato}) ,
\begin{equation*}
\frac{\dif }{\dif t}q_n(t) = -\sum_{i\in\mathcal{I}} \lambda_i(t) q_n(t) + \sum_{i\in\mathcal{I}} \lambda_i(t) q_{n-i}(t),\ \ \ t\ge0,\ n\in\mathbb{N},
\end{equation*}
with initial condition $q_n(0) = P\{T_n>0\} = 1, \ n\ge1$ and by assuming $q_n(t)=0,\ t\ge0,\ n\le0$.

Finally, from (\ref{generatriceTempiPrimoPassaggioRipartizione}) we can derive the generating function for the moments of order $r>0$ of the random times $T_n$,
\begin{equation}\label{funzioneGeneratriceMomentiTempoPrimoPassaggio}
\sum_{n=1}^\infty u^n \mathbb{E} T^r_n = \frac{u\,r}{1-u}\int_0^\infty t^{r-1}G_t(u)\dif t.
\end{equation}
Note that one can directly obtain the last term of \eqref{funzioneGeneratriceMomentiTempoPrimoPassaggio} by using formula (\ref{generatriceTempiPrimoPassaggio}). In the case of constant rates, formula (\ref{funzioneGeneratriceMomentiTempoPrimoPassaggio}) permits us to obtain that
$$ \sum_{n=1}^\infty u^n \mathbb{E} T^r_n  = \frac{u\,\Gamma(r+1)}{1-u}\Biggl(\sum_{i\in\mathcal{I}} \lambda_i(1-u^i)\Biggr)^{-r}.$$

\subsection{Limit results}\label{sezioneLimiti}

We here show some limit results of the type of the law of large numbers (or the ergodic theorem) and the central limit theorem. After considering the limit as the time goes to $\infty$ we consider the case of rate functions exploding to infinite to derive the convergence of a Skellam process to a Gaussian process. Again, we assume $\mathcal{I}\subset \mathbb{R}$.

\begin{theorem}\label{teoremaConvergenzaSkellamGeneralizzato}
Let $S\sim NHGSP (\lambda_i,i\in\mathcal{I})$. Let $f\in C^1\Bigl([0,\infty), [0,\infty)\Bigr)$ be a non-decreasing function such that $f(t)\longrightarrow\infty$ as $t\longrightarrow\infty$. 
\begin{itemize}
\item[($i$)] Let $\Lambda_i(t) /f(t)\longrightarrow \mu_i\ge0$, as $t\longrightarrow \infty$, $ i\in\mathcal{I}$. Then,
\begin{equation}\label{convergenzaMediaLeggeGrandiNumeri}
\frac{S(t)}{f(t)}\stackrel{p,\, L^1}{ \substack{\xrightarrow{\hspace*{1.7cm}}\\t\longrightarrow \infty }} \sum_{i\in\mathcal{I}}i \mu_i.
\end{equation}
%If, in addition, %$f$ is non-decreasing and 
%$\sum_{k=1}^\infty \Lambda_i(k,k+1)/f(k)^2<\infty,\ \forall\ i$, then the convergence is $a.s.$

\item[($ii$)] Let $\Lambda_i(t) = \mu_i(t) +\sigma^2_i (t) $ such that 
\begin{equation}\label{ipotesiConvergenzaNormaleMuSigma}
\frac{\mu_i(t)}{\sqrt{f(t)}}\substack{\xrightarrow{\hspace*{1.5cm}}\\t\longrightarrow \infty }\mu_i\in\mathbb{R},\ \ \frac{\sigma^2_i(t)}{f(t)}\substack{\xrightarrow{\hspace*{1.5cm}}\\t\longrightarrow \infty }\sigma^2_i\ge0,\ \forall\ i, \ \text{ and }\ \ \sum_{i\in\mathcal{I}} i \sigma^2_i = 0.
\end{equation}
%\begin{equation}\label{ipotesiConvergenzaTeoremaLimiteCentrale}
%\frac{\mu_i(t)}{\sqrt{f(t)}}\longrightarrow \mu_i\in\mathbb{R}, \ \ \frac{\sigma_i(t)}{f(t)}\longrightarrow \sigma_i>0,\ \text{ with }\  \sum_{i\in\mathcal{I}} i \sigma_i = 0,
%\end{equation}
%where the limits are taken as $t\longrightarrow \infty$ . 
Then,
\begin{equation}
\frac{S(t)}{\sqrt{f(t)}}\stackrel{d}{ \substack{\xrightarrow{\hspace*{1.7cm}}\\t\longrightarrow \infty }}Z\sim\mathcal{N}\Biggl( \sum_{i\in\mathcal{I}}i \mu_i, \sum_{i\in\mathcal{I}} i^2 \sigma^2_i\Biggr).
\end{equation}
\end{itemize}
\end{theorem}

Note that the last hypothesis in (\ref{ipotesiConvergenzaNormaleMuSigma}) implies that the process $S$ performs both positive and negative jumps. If $\sigma^2_i = 0\ \forall\ i$ then point ($ii$) reduces to ($i$).

\begin{proof}
($i$) Convergence in probability: it is sufficient to consider the representation (\ref{SkellamGeneralizzatoInTerminiDiPoisson}) and that under the given assumptions, it is known that for each $i\in\mathcal{I}$,  $N_i(t)/f(t)\stackrel{p}{\longrightarrow} \mu_i$ (it can be easily derived by studying the limit of the moment generating function $\mathbb{E}\text{exp}\bigl(-\gamma N_i(t)/f(t) \bigr) =\text{exp}\Bigl(- \Lambda_i(t)\bigl(1-e^{\gamma/f(t)}\bigr) \Bigr) $).

Convergence in mean: if $\mu_i = 0$, then $\mathbb{E}N_i(t)/f(t) \longrightarrow 0$ and the result is obvious. If $\mu_i>0$, then $\Lambda_i(t)\longrightarrow\infty$ as $t\rightarrow\infty$. Now, for $t\ge0$ there exists $n_t>0$ such that $n_t <\Lambda_i(t) \le n_t+1$ and 
\begin{align*}
	\mathbb{E}\Bigg|\frac{N_i(t)}{\Lambda_i(t)}- 1\Bigg| &= \sum_{n=0}^\infty \Big|\frac{n}{\Lambda_i(t)}- 1\Big| e^{-\Lambda_i(t)}\frac{\Lambda_i(t)^n}{n!}\\
	& =  e^{-\Lambda_i(t)}\left(\sum_{n=0}^{n_t} \Big(1-\frac{n}{\Lambda_i(t)}\Big)\frac{\Lambda_i(t)^n}{n!} + \sum_{n=n_t+1}^\infty \Big(\frac{n}{\Lambda_i(t)}-1\Big)\frac{\Lambda_i(t)^n}{n!}\right)\\
	& = 2 e^{-\Lambda_i(t)}\frac{\Lambda_i(t)^{n_t}}{n_t!} \substack{\xrightarrow{\hspace*{1.5cm}}\\t\longrightarrow \infty } 0
\end{align*}
%$$\mathbb{E}\Bigg|\frac{N_i(t)}{\Lambda_i(t)} - 1\Bigg|=\mathbb{E}\Bigg|\frac{N_i(t)}{\mathbb{E}N_i(t)} - 1\Bigg|\le2$$ 
%and thus, by means of the dominated convergence theorem 
%$$\lim_{t\rightarrow\infty}\mathbb{E}\Bigg|\frac{N_i(t)}{\Lambda_i(t)}- 1\Bigg| = \sum_{n=0}^\infty \lim_{t\rightarrow\infty} \Big|\frac{n}{\Lambda_i(t)}- 1\Big| e^{-\Lambda_i(t)}\frac{\Lambda_i(t)^n}{n!} = 0.$$
where the limit follows by keeping in mind that $n_t\approx \Lambda_i(t)$ (and thus also $n_t\longrightarrow\infty $ as $t\rightarrow\infty$) and Stirling forumla $n! \sim \sqrt{2\pi n}(n/e)^n$. Hence, $N_i(t)/f(t)$ converges in mean to $\mu_i$ for each $i\in\mathcal{I}$ and this, by keeping in mind the definition (\ref{SkellamGeneralizzatoInTerminiDiPoisson}), yields the $L^1$-convergence in (\ref{convergenzaMediaLeggeGrandiNumeri}).
\\
%Almost sure convergence: again, it is sufficient to prove that $N_i(t)/f(t)\stackrel{a.s}{\longrightarrow}	\mu_i,\ \forall\ i$.
%First, we rewrite that $N_i(n)/f(n) = \sum_{k=0}^{n-1} N_i(k, k+1)/ f(n)$ for $n\in\mathbb{N}$, and we consider the modified process $\sum_{k=0}^\infty \Bigl( N_i(k, k+1)-\mathbb{E}N_i(k,k+1)\Bigr)/ f(k)$. This converges almost surely thanks to the Kolmogorov's convergence criterion (see Theorem 6.5.2 of \cite{G2005}) since $\sum_{k=0}^\infty \text{Var}\Bigl(  N_i(k, k+1)/ f(k) \Bigr) = \sum_{k=0}^\infty \Lambda_i(k,k+1)/f(k)^2<\infty$ by hypothesis. Finally, Kronecker Lemma (see Lemma 6.5.1 of \cite{G2005}) implies that 
%\begin{equation*}
%\frac{1}{f(n)}\sum_{k=0}^{n-1} \Bigl( N_i(k, k+1)-\mathbb{E}N_i(k,k+1)\Bigr)\stackrel{a.s.}{ \substack{\xrightarrow{\hspace*{1.7cm}}\\n\longrightarrow\infty}}0.
%\end{equation*}
%Finally, by observing that $\sum_{k=0}^{n-1} \mathbb{E}N_i(k,k+1)/f(n) = \Lambda_i(n)/f(n)\longrightarrow \mu_i$ we obtain $N_i(n)/f(n)\stackrel{a.s}{\longrightarrow}	\mu_i$.
%To conclude the proof of ($i$) we extend the result to $t\ge0$:
%\begin{align*}
%\frac{N_i(t)}{f(t)} = \sum_{k=0}^{\floor{t}-1} \frac{N_i(k,k+1)}{f(\floor{t})}\frac{f(\floor{t})}{f(t)} + \frac{N_i(\floor{t},t)}{f(t)}\stackrel{a.s.}{ \substack{\xrightarrow{\hspace*{1.7cm}}\\t\longrightarrow\infty}}\mu_i  + 0  =\mu_i,
%\end{align*}
%where $N_i(\floor{t},t)/f(t)\stackrel{a.s}{\longrightarrow}	0$ since $\sum_{k=1}^n N_i(k, k+1)/ f(k)$ converges $a.s.$
%\\

($ii$) We study the limit of the moment generating function, which, in light of (\ref{funzioneGeneratriceMomentiSkellamGeneralizzatoNonOmogeneo}) and the hypotheses (\ref{ipotesiConvergenzaNormaleMuSigma}), is, for suitable real $\gamma$,
$$\mathbb{E}e^{\gamma S(t)/\sqrt{f(t)}} = \text{exp}\Biggl(- \sum_{i\in\mathcal{I}}\Bigl(\mu_i(t) +\sigma^2_i(t)\Bigr)\Bigl(1-e^{i \gamma/\sqrt{f(t)}}\Bigr) \Biggr). $$ 

We now compute the limit of the exponent above.
\begin{align*}
- \lim_{t\rightarrow \infty} &\sum_{i\in\mathcal{I}}\Bigl(\mu_i(t) +\sigma^2_i (t)\Bigr)\Bigl(1-e^{i \gamma/\sqrt{f(t)}}\Bigr)\\
&= \lim_{t\rightarrow \infty}  \sum_{i\in\mathcal{I}}\Bigl(\mu_i(t) +\sigma^2_i(t)\Bigr)\sum_{k=1}^\infty \Biggl( \frac{i \gamma}{\sqrt{f(t)}}\Biggr)^k \frac{1}{k!} \\
& = \lim_{t\rightarrow \infty} \sum_{i\in\mathcal{I}} \Biggl( i\gamma \frac{\mu_i(t)}{\sqrt{f(t)}} + \sum_{k=2}^\infty \frac{(i \gamma)^k}{k!}\frac{\mu_i(t)}{\sqrt{f(t)^k}} + i\gamma \frac{\sigma^2_i(t)}{\sqrt{f(t)}} +  \frac{i^2\gamma^2}{2}\frac{\sigma^2_i(t)}{f(t)} +  \sum_{k=3}^\infty \frac{(i \gamma)^k}{k!}\frac{\sigma_i^2(t)}{\sqrt{f(t)^k}} \Biggr)\\
& =\gamma \sum_{i\in\mathcal{I}}i \mu_i +  \frac{\gamma^2}{2}\sum_{i\in\mathcal{I}}i^2 \sigma^2_i, %\label{passaggioLimiteSommaVarianze0}
\end{align*}
where in the last step we used the hypotheses (\ref{ipotesiConvergenzaNormaleMuSigma}) (after the interchange of the limit and the series).
\end{proof}

In view of Theorem \ref{teoremaConvergenzaSkellamGeneralizzato} we can also derive the following convergence results inspired by the hydrodynamic limit (also known as Kac's limit). This permits us to obtain also the weak convergence of the whole process.

\begin{corollary}\label{corollarioConvergenzaSkellamProcessoGaussiano}
Let $\alpha\ge1$ and $S_\alpha\sim NHGSP \Bigl(\lambda_i(\cdot, \alpha),i\in\mathcal{I}\Bigr)$.
\begin{itemize}
\item[($i$)] Let $\Lambda_i(t;\alpha)$ such that $\Lambda_i(t;\alpha) /\alpha\longrightarrow \mu_i(t)\ge0,$ as $\alpha\longrightarrow \infty, \ \forall\ t\ge0, \ i\in\mathcal{I}$. Then,
\begin{equation}
\frac{S_\alpha(t)}{\alpha}\stackrel{p,\,L^1}{ \substack{\xrightarrow{\hspace*{1.7cm}}\\\alpha\longrightarrow\infty}} \sum_{i\in\mathcal{I}}i \mu_i(t), \ \ \ t\ge0.
\end{equation}

\item[($ii$)] Let $\Lambda_i(t;\alpha) = \int_0^t \lambda_i^{\mu}(s;\alpha)\dif s  +\int_0^t\lambda_i^{\sigma}(s;\alpha)\dif s = \mu_i(t;\alpha) +\sigma^2_i(t;\alpha)$ with suitable real functions $ \lambda_i^{\mu}$ and non-negative functions $\lambda_i^{\sigma}$ such that for $t\ge0$,
\begin{equation*}
\frac{\mu_i(t;\alpha)}{\sqrt{\alpha}}\substack{\xrightarrow{\hspace*{1.4cm}}\\\alpha\longrightarrow\infty} \mu_i(t)\in\mathbb{R}, \ \ \frac{\sigma_i^2(t;\alpha)}{\alpha}\substack{\xrightarrow{\hspace*{1.4cm}}\\\alpha\longrightarrow\infty} \sigma_i^2(t)\ge0,\ \text{ with }\  \sum_{i\in\mathcal{I}} i \sigma^2_i (t;\alpha) = 0.
\end{equation*}
Then, for $n\in\mathbb{N}$ and $0\le t_1<\dots< t_n$,
\begin{equation}\label{convergenzaFinitoDimensionaliSkellamLimiteKac}
\Biggl(\frac{S_\alpha(t_1)}{\sqrt{\alpha}}, \dots, \frac{S_\alpha(t_n)}{\sqrt{\alpha}}\Biggr)\stackrel{d}{ \substack{\xrightarrow{\hspace*{1.7cm}}\\\alpha\longrightarrow\infty}} \Bigl(Z(t_1),\dots,Z(t_n)\Bigr)
\end{equation}
where $Z$ is a Gaussian process with independent increments and such that $Z(t)\sim\mathcal{N}\biggl( \sum_{i\in\mathcal{I}}i \mu_i(t), \sum_{i\in\mathcal{I}} i^2 \sigma_i^2(t)\biggr), \ t\ge0$.
\\
If, in addition, $\sum_i i \lambda^\mu_i(t;\alpha) =0, \forall \ t,\alpha,$ and for each $i\in\mathcal{I}$,
\begin{equation}\label{ipotesiConvergenzaDeboleAProcessoGaussiano}
\exists \ \ M_i>0\ \ s.t.\ \ \lambda_i^\mu(t;\alpha)\le\sqrt{\alpha}M_i, \ \ \lambda_i^\sigma(t;\alpha)\le\alpha M_i,
\end{equation}
then $S_\alpha/\sqrt{\alpha} \implies Z$ on $\mathcal{D}[0,\infty)$ (i.e. $S_\alpha/\sqrt{\alpha}$ converges weakly to $Z$ as $\alpha\rightarrow \infty$).
\end{itemize}
\end{corollary}

We note that here the convergence in $\mathcal{D}[0,\infty)$ is the so-called weak convergence in the space of the real c\`{a}dl\`{a}g functions endowed with the Skorokhod topology and in ($ii$) we request the process to have zero mean. The hypothesis (\ref{ipotesiConvergenzaDeboleAProcessoGaussiano}) is stronger than one could require; indeed, we just need the tightness of $S_\alpha$. However, the given hypotheses are sufficient for the convergence of the homogeneous process with rates $\lambda_i(t;\alpha) = \alpha\lambda_i\ \forall\ i$ such that $\sum_{i\in\mathcal{I}} i \lambda_i(t;\alpha)=0\ \forall\ t,\alpha$. In this case $S_\alpha/\sqrt{\alpha}$ converges weakly to a scaled Brownian motion.

\begin{proof}
($i$) easily follows from point ($i$) of Theorem \ref{teoremaConvergenzaSkellamGeneralizzato}.

($ii$) We recall that the increments of $S$ are independent and they behave like a Skellam-type process with rate functions connected to the original $\lambda_i$, see (\ref{incrementoSkellamGeneralizzato}). Now, the convergence (\ref{convergenzaFinitoDimensionaliSkellamLimiteKac}) follows from ($ii$) of Theorem \ref{teoremaConvergenzaSkellamGeneralizzato}. 

To prove the weak convergence we can prove that $S_\alpha/\sqrt{\alpha}$ is tight (see Theorem 13.15 and formula (13.14) of \cite{B1999}), thus we show that $\exists \ \beta\ge0,\gamma>1/2$ and $F$ non-decreasing continuous function such that $\mathbb{E}\big|S_\alpha(t)-S_\alpha(s) \big|^{2\beta}\big|S_\alpha(s)-S_\alpha(r) \big|^{2\beta}/\alpha^{2\beta}\le \big|F(t)-F(s) \big|^{2\gamma},\ \forall\ 0\le r\le s\le t$ and $\alpha>0$. Note that under the hypotheses at the end of the statement $\mathbb{E}S(t) = 0$. Now, we show the following inequality
\begin{align}
\mathbb{E}\Big|\frac{S_\alpha(t)-S_\alpha(s) }{\sqrt{\alpha}}\Big|^{2}\Big|\frac{S_\alpha(s)-S_\alpha(r) }{\sqrt{\alpha}}\Big|^{2} &= \frac{1}{\alpha^2}\mathbb{E}\Bigl( S_\alpha(t)-S_\alpha(s)\Bigr)^2 \mathbb{E}\Bigl( S_\alpha(s)-S_\alpha(r)\Bigr)^2  \nonumber\\
& = \frac{1}{\alpha^2} \sum_{i\in\mathcal{I}} i^2 \Lambda_i(s,t;\alpha) \sum_{i\in\mathcal{I}} i^2 \Lambda_i(r,s;\alpha) \nonumber\\
& \le \frac{1}{\alpha^2} \sum_{i\in\mathcal{I}} i^2 (\sqrt{\alpha}+\alpha)M_i (t-s) \sum_{i\in\mathcal{I}} i^2 (\sqrt{\alpha}+\alpha) M_i (s-r)  \label{passaggioIpotesiConvergenzaDeboleAProcessoGaussiano} \\
& \le \Bigl(\frac{1}{\sqrt{\alpha}} +1\Bigr)^2 M^2(t-r)^2 \nonumber\\
& \le 4M^2(t-r)^2,\nonumber
\end{align}
where in step (\ref{passaggioIpotesiConvergenzaDeboleAProcessoGaussiano}) we used (\ref{ipotesiConvergenzaDeboleAProcessoGaussiano}) and $M = \sum_{i\in\mathcal{I}} i^2M_i<\infty$. Hence $S_\alpha$ is tight and this concludes the proof.
\end{proof}

\subsection{Homogeneous case}\label{sottosezioneCasoOmogeneo}

We now derive further properties for the homogeneous generalized Skellam process, i.e. when $\lambda_i(t) = \lambda_i,\ t\ge0,\ \forall \ i$. We denote this process by writing $S\sim HGSP (\lambda_i, i\in \mathcal{I} )$.

In this case $S$ is a L\'{e}vy process with L\'{e}vy measure $\nu(x) = \sum_{i\in\mathcal{I}} \lambda_i\delta_{\{i\}}(x),\ x\in\mathbb{R}$, where $\delta_{\{a\}} $ is the Dirac delta function centered in $a\in\mathbb{R}$. This readily follows from (\ref{SkellamGeneralizzatoInTerminiDiPoisson}).

Furthermore, since the functions $\Lambda_i(t) = \lambda_i t$ are proportional to time, 
\begin{equation}\label{sommaSkellamOmogeneoConTempoModificato}
\sum_{k=1}^K S(t X_k ) = S\Bigg(t\sum_{k=1}^K X_k\Bigg), \ \ \ t\ge0,
\end{equation}
 where $X_1,\dots, X_K$ are non-negative random variables.

\begin{remark}[Skellam process as compound Poisson]\label{proposizioneSkellamGeneralizzatoPoissonComposto}
Let $S\sim HGSP (\lambda_i, i\in \mathcal{I} )$ and $\{X_k\}_{k\in\mathbb{N}}$ be a sequence of independent and identically distributed random variables, copies of a r.v. $X$ such that $P\{X = i\} = \lambda_i / \sum_{j\in\mathcal{I}} \lambda_j$ for $i\in\mathcal{I}$. Then, it is useful to express $S$ as a compound Poisson process,
\begin{equation}\label{SkellamGeneralizzatoComePoissonComposto}
S(t) = \sum_{k = 1}^{N(t)} X_k,\ \ \ t\ge0,
\end{equation}
where $N$ is an independent Poisson process with rate $\sum_{i\in\mathcal{I}} \lambda_i$.  Indeed, the probability generating function of $X$ is, for suitable $u$ in a neighborhood of $0$,
$$\mathbb{E} u^X = \frac{\sum_{i\in\mathcal{I}} \lambda_i u^i}{ \sum_{i\in\mathcal{I}} \lambda_i}. $$
Now, by using the generating function of a compound Poisson, we derive, for $t\ge0$,
\begin{equation*}
\mathbb{E} u^{\sum_{k=1}^{N(t)} X_k} = \text{exp}\Biggl(-t \sum_{i\in\mathcal{I}}  \lambda_i \Bigl(1-\frac{\sum_{i\in\mathcal{I}} \lambda_i u^i}{ \sum_{i\in\mathcal{I}} \lambda_i}\Bigr) \Biggr) = \text{exp}\Biggl(-t \sum_{i\in\mathcal{I}}  \lambda_i -t\sum_{i\in\mathcal{I}} \lambda_iu^i  \Biggr),
\end{equation*}
which coincides with (\ref{funzioneGeneratriceSkellamGeneralizzatoNonOmogeneo}). 

We point out, that if the rate functions are not constant but satisfy $\lambda_i =p_i\lambda, i\in\mathcal{I}$, with $\lambda$ being a continuous and integrable non-negative function and $\sum_{i}p_i = 1$, then expression \eqref{SkellamGeneralizzatoComePoissonComposto} holds with $N$ being a non-homogeneous Poisson process and $X$ such that $P\{X=i\} = p_i,\ i\in\mathcal{I}$. \hfill $\diamond$
\end{remark}

%\begin{remark}[First passage time]
Remark \ref{proposizioneSkellamGeneralizzatoPoissonComposto} permits us to describe the generalized Skellam process in terms of a random walk with a random number of steps. Thus, we can describe some properties, like the first passage time or the sojourn time of the process in terms of those of random walks.
For instance, by assuming $T_n = \inf\{t\ge0\,:\,S(t)\ge n\}$, by means of classical arguments on the compound Poisson processes we obtain that, for $t\ge0$,
$$P\{T_n \le t\} =\sum_{k=0}^\infty P\{N(t) = k\}P\{T_n^{X} \le k\}, $$
where $T_n^X = \inf\{m\in\mathbb{N}\,:\,\sum_{k=1}^m X_k \ge n\}$ is the first passage time of the random walk with steps $X_k$ given in Remark \ref{proposizioneSkellamGeneralizzatoPoissonComposto}. Also the mean follows, $\mathbb{E} T_n = \mathbb{E} T_n^X /\sum_{i\in\mathcal{I}}\lambda_i$.
%\end{remark}

We point out that the Bernoulli decomposition in Proposition \ref{proposizioneCoppiaSkellamScompostoBernoulli}, in the homogeneous case readily comes from the compound Poisson representation and its thinning property.%, which we recall in the following Lemma.
%
%\begin{lemma}\label{lemmaDecomposizioneBernoulliPoissonComposto}
%Let $Z$ be a compound Poisson process such that  $Z(t) = \sum_{k=1}^{N(t)}  X_k,\ t\ge0$, where $N$ is an independent Poisson process of rate $\lambda>0$ and $X_1,\dots$ are i.i.d. random variables. Let $\{B_k\}_{k\in\mathbb{N}}$ be a sequence of i.i.d. Bernoulli random variables with parameter $p\in (0,1)$.
%Then, the Bernoulli decomposition of $Z$ produces two independent compound Poisson processes,
%\begin{equation*}\label{decomposizioneSkellamOmogeneoTramitePoissonComposto}
%\sum_{k=1}^{N(t)} X_kB_k \stackrel{d}{=}\sum_{k=1}^{N_p(t)} X_k \ \ \text{ and }\ \ \ \sum_{k=1}^{N(t)} X_k(1-B_k) \stackrel{d}{=}\sum_{k=1}^{N_{1-p}(t)} X_k,\ \ \ t\ge0,
%\end{equation*}
%where $N_q$ is a Poisson process with rate $\lambda q$.
%\end{lemma}
%The interested reader can find the proof of Lemma \ref{lemmaDecomposizioneBernoulliPoissonComposto} in the Appendix \ref{appendiceDecomposizioneBernoulliPoissonComposto}. Furthermore, Lemma \ref{lemmaDecomposizioneBernoulliPoissonComposto} can be easily extended to a decomposition into $H$ independent subprocesses by means of a Multinomial distribution (indeed, the vector $(B_k, 1-B_k)$ is a Multinomial distribution of size $2$ with parameter $(p,1-p)$).
%\\

We conclude this section by showing an approximation of the generalized Skellam process, generalizing the binomial approximation of the Poisson case.
\begin{proposition}\label{proposizioneApprossimazioneBinomialeSkellamGeneralizzato}
Let $Z_n = \Big\{Z_n(t) = \sum_{k=1}^{\floor{a_n t}} X_k^{(n)} \Big\}_{t\ge0}$ where $X_k^{(n)},\dots$ are independent random variables $\forall\ n,k$ and such that
\begin{equation}
 X_k^{(n)} = \begin{cases}\begin{array}{l l}
i\in\mathcal{I}, & p_{ki}^{(n)},\\
0, & 1- \sum_{i\in\mathcal{I}} p_{ki}^{(n)},
\end{array}\end{cases}
\end{equation}
where $p_{ki}^{(n)}\in(0,1)\ \forall\ i$ and $\sum_{i\in\mathcal{I}} p_{ki}^{(n)}<1$.
If
\begin{equation}\label{ipotesiApprossimazioneSkellam}
\sum_{k=1}^{\floor{a_n t}}p_{ki}^{(n)}\substack{\xrightarrow{\hspace*{1.4cm}}\\ n\longrightarrow\infty}\lambda_i t,\ \ t\ge0, \ \text{ and }\ \max_{0\le k\le a_n} p_{ki}^{(n)}\substack{\xrightarrow{\hspace*{1.4cm}}\\ n\longrightarrow\infty}0 \ \forall \ i\in\mathcal{I},
\end{equation}
then, for $0\le t_1<\dots< t_h$,
\begin{equation}\label{convergenzaFinitoDimensionaliApprossimazioneSkellam}
\Bigl(Z_n(t_1),\dots, Z_n (t_h)\Bigr)\stackrel{d}{\longrightarrow} \Bigl(S(t_1),\dots, S (t_h)\Bigr),
\end{equation}
with $S\sim HGSP (\lambda_i, i\in \mathcal{I} )$.

If, in addition, $\mathbb{E}X_k^{(n)} = 0\ \forall\ n,k$ and $\forall\ i\in\mathcal{I}\ \exists\ F_i$ non-decreasing, continuous functions s.t. $\forall\ 0\le r\le s\le t,$
\begin{equation}\label{ipotesiConvergenzaDeboleApprossimazioneBinomiale}
 \sum_{k=\floor{a_n s}+1}^{\floor{a_n t}}p_{ki}^{(n)}\sum_{k=\floor{a_n r}+1}^{\floor{a_n s}}p_{kj}^{(n)} \le \bigl(F_i(t)-F_i(r)\bigr)\bigl(F_j(t)-F_j(r)\bigr)\ \forall\ i,j,n,
\end{equation}
then $Z_n \implies S$ on $\mathcal{D}[0,\infty)$.
\end{proposition}

We point out that in the case of $p_{ki}^{(n)} = \lambda_i / n\ \forall\ i,k$ (with $n$ sufficiently large), $a_n = n$ and $\sum_{i\in\mathcal{I}}i\lambda_i = 0$, the hypotheses of Proposition \ref{proposizioneApprossimazioneBinomialeSkellamGeneralizzato} hold. We briefly show how to check hypothesis (\ref{ipotesiApprossimazioneSkellam}). If $t-r< 1/n$, then either $ \floor{n r} = \floor{n s}$ or $\floor{n s} = \floor{n t}$ so $\sum_{k=\floor{n s}+1}^{\floor{n t}}p_{ki}^{(n)}\sum_{k=\floor{n r}+1}^{\floor{n s}}p_{kj}^{(n)} = 0\le\max_{i\in\mathcal{I}}\lambda_i^2 4(t-r)^2$. If $t-r\ge1/n$, then
\begin{align*}
\sum_{k=\floor{n s}+1}^{\floor{n t}}p_{ki}^{(n)}\sum_{k=\floor{n r}+1}^{\floor{n s}}p_{kj}^{(n)}& =  \sum_{k=\floor{n s}+1}^{\floor{n t}} \frac{\lambda_i}{n}\sum_{k=\floor{n r}+1}^{\floor{n s}} \frac{\lambda_j}{n} = \frac{\lambda_i}{n}\bigl(\floor{n t} - \floor{n s}\bigr) \frac{\lambda_j}{n}\bigl(\floor{n s} - \floor{n r}\bigr) \\
&\le  \frac{\lambda_i}{n}\bigl(n t - (n s-1)\bigr)\frac{\lambda_j}{n}\bigl(n s - (n r-1)\bigr) \le \lambda_i\Bigl(t-r+\frac{1}{n}\Bigr) \lambda_j\Bigl(t-r+\frac{1}{n}\Bigr) \\
&\le \lambda_i 2(t-r)\lambda_j 2 (t-r)  \le \max_{i\in\mathcal{I}}\lambda_i^2 4(t-r)^2.
\end{align*}
Thus, $Z_n\implies S$ (with $\mathbb{E}S(t) = 0\ \forall\ t$). In this case, if $\mathcal{I} = \{1\}$ we have the binomial approximation of the Poisson process.

\begin{proof}
First, note that $Z_n$ has independent increments, indeed, for $0\le s<t,\ Z_n(t)- Z_n(s)  = \sum_{k=\floor{a_n s}+1}^{\floor{a_n t}} X_{k}^{(n)}$. Now, to prove the convergence of the finite dimensional distributions (\ref{convergenzaFinitoDimensionaliApprossimazioneSkellam}) it is sufficient to prove that $Z_n(t) - Z_n(s) \stackrel{d}{\longrightarrow}S(t)-S(s)$. We prove it by showing the convergence of the probability generating function,
\begin{align}
\mathbb{E} u^{Z_n(t)- Z_n(s)} &= \prod_{k=\floor{a_n s}+1}^{\floor{a_n t}}\mathbb{E}u^{X_k^{(n)}} =  \prod_{k=\floor{a_n s}+1}^{\floor{a_n t}}\Biggl( \sum_{ i\in\mathcal{I} } u^ip_{ki}^{(n)}+ 1- \sum_{ i\in\mathcal{I} } p_{ki}^{(n)}\Biggr) \nonumber\\
&= \exp\Biggl( \sum_{k=\floor{a_n s}+1}^{\floor{a_n t}} \ln\Bigl( 1+ \sum_{ i\in\mathcal{I} } (u^i-1) p_{ki}^{(n)} \Bigr) \Biggr) \nonumber\\
& \sim \exp\Biggl( \sum_{k=\floor{a_n s}+1}^{\floor{a_n t}}  \sum_{ i\in\mathcal{I} } (u^i-1) p_{ki}^{(n)}  \Biggr) \label{passaggioLimiteApprossimazioneSkellam}\\
& =  \exp\Biggl(   \sum_{ i\in\mathcal{I} } (u^i-1) \sum_{k=\floor{a_n s}+1}^{\floor{a_n t}} p_{ki}^{(n)}  \Biggr) \nonumber\\
&\longrightarrow e^{\sum_{ i\in\mathcal{I} } \lambda_i (t-s) (u^i-1)}, \ \ \ \text{as }n\longrightarrow\infty,
\end{align}
which is the probability generating function of the increment $S(t)-S(s)$, see (\ref{funzioneGeneratriceSkellamGeneralizzatoNonOmogeneo}). Note that in step (\ref{passaggioLimiteApprossimazioneSkellam}) we considered the approximation of the logarithm by keeping in mind that $\sum_{i\in\mathcal{I}} (u^i-1)p_{ki}^{(n)}\longrightarrow0$ as $n\longrightarrow\infty$, which follows from (\ref{ipotesiApprossimazioneSkellam}), since $p_{ki}^{(n)}\longrightarrow0 \ \forall \ i,k$.
\\

Now, in order to prove the weak convergence we show that $Z_n$ is tight. As described in the proof of Corollary \ref{corollarioConvergenzaSkellamProcessoGaussiano}, it is sufficient to show that $\exists \ \alpha\ge0,\beta>1/2$ and $F$ non-decreasing continuous function such that $\mathbb{E}\big|Z_n(t)-Z_n(s) \big|^{2\alpha}\big|Z_n(s)-Z_n(r) \big|^{2\alpha}\le \big|F(t)-F(s) \big|^{2\beta},\ \forall\ 0\le r\le s\le t$ and $n$. 
\begin{align}
\mathbb{E}\big|Z_n(t)-Z_n(s) \big|^2\big|Z_n(s)-Z_n(r) \big|^2 & = \mathbb{E}\Biggl( \sum_{k=\floor{a_n s}+1}^{\floor{a_n t}} X_{k}^{(n)}\Biggr)^2  \mathbb{E}\Biggl( \sum_{k=\floor{a_n r}+1}^{\floor{a_n s}} X_{k}^{(n)}\Biggr)^2  \nonumber\\
&=  \Biggl(\sum_{k=\floor{a_n s}+1}^{\floor{a_n t}}\mathbb{E}\bigl(X_{k}^{(n)}\bigr)^2 \Biggr)\Biggl( \sum_{k=\floor{a_n r}+1}^{\floor{a_n s}} \mathbb{E}\bigl(X_{k}^{(n)}\bigr)^2\Biggr)  \nonumber\\
& =\sum_{i,j\in\mathcal{I}}i^2j^2 \sum_{k=\floor{a_n s}+1}^{\floor{a_n t}}p_{ki}^{(n)} \sum_{k=\floor{a_n r}+1}^{\floor{a_n s}}p_{kj}^{(n)} \nonumber\\
&\le \sum_{i,j\in\mathcal{I}}i^2j^2  \bigl(F_i(t)-F_i(r)\bigr)\bigl(F_j(t)-F_j(r)\bigr) \label{passaggioConvergenzaDeboleApprossimazioneSkellam} \\
& = \sum_{i\in\mathcal{I}}i^2 \bigl(F_i(t)-F_i(r)\bigr)\sum_{j\in\mathcal{I}}j^2 \bigl(F_j(t)-F_j(r)\bigr)  \nonumber\\
& =\bigl(F(t)-F(r)\bigr)^2 \nonumber
\end{align}
where $F = \sum_{i\in\mathcal{I}}i^2F_i$ is a non-decreasing and continuous function (since it is sum of non-decreasing continuous functions). Note that in (\ref{passaggioConvergenzaDeboleApprossimazioneSkellam}) we used the hypothesis (\ref{ipotesiConvergenzaDeboleApprossimazioneBinomiale}).
\end{proof}

\section{Fractional integral of the generalized Skellam process}\label{sezioneIntegraleSkellam}

We now focus our attention on the fractional integral of the non-homogeneous generalized Skellam process. For a suitable function $f$, $I^\alpha f(t) = \int_0^t (t-s)^{\alpha - 1} f(s)\dif s/ \Gamma(\alpha)$, with $t\ge0, \alpha>0$, denotes the fractional derivative of $f$ of order $\alpha$. Now, let $S\sim NHGSP (\lambda_i, i\in \mathcal{I} )$, we study the fractional integral of order $\alpha>0$ of $S$, meaning the process $S^\alpha = \big\{S^\alpha(t)\big\}_{t\ge0}$ such that
\begin{equation}\label{integraleFrazionarioSkellamGeneralizzato}
S^\alpha (t) = I^\alpha S(t) =  \frac{1}{\Gamma(\alpha)} \int_0^t (t-s)^{\alpha-1} S(s)\dif s,\ \ \ t\ge0,\ \alpha>0.
\end{equation}
By denoting with $N^\alpha$ the fractional integral of the Poisson process (i.e. (\ref{integraleFrazionarioSkellamGeneralizzato}) with $\mathcal{I} = \{1\}$), from (\ref{SkellamGeneralizzatoInTerminiDiPoisson}), we derive that 
\begin{equation}\label{integraleSkellamGeneralizzatoInTerminiDiPoisson}
S^{\alpha}(t) = \sum_{i\in\mathcal{I}} i N_i^\alpha(t),\ \ \ t\ge0,\ \alpha>0,
\end{equation}
where the processes $N_i^\alpha s$ are all independent.

We point out that some characteristics of the fractional integral of the homogeneous Poisson process have been studied in the literature, see for instance \cite{OP2013}.

\begin{proposition}
The process $S^\alpha$ in (\ref{integraleFrazionarioSkellamGeneralizzato}) has the following moments:
\begin{equation}\label{momentiPrimoSecondoIntegraleSkellamGeneralizzatoNonOmogeneo}
\mathbb{E}S^\alpha(t) = \sum_{i\in\mathcal{I}} i I^\alpha \Lambda_i(t) =  \sum_{i\in\mathcal{I}} i I^{\alpha+1} \lambda_i(t), \ \ \  \text{Var}S^\alpha(t) = \frac{2\Gamma(2\alpha)}{\Gamma(\alpha)\Gamma(\alpha+1)}\sum_{i\in\mathcal{I}} i^2 I^{2\alpha}\Lambda_i(t),  
\end{equation} 
and for $0\le s,t$,
\begin{equation}\label{covarianzaIntegraleSkellamGeneralizzato}
 \text{Cov}\Bigl( S^\alpha(s),S^\alpha(t)\Bigr) = \sum_{i\in\mathcal{I}} \frac{i^2}{\Gamma(\alpha)\Gamma(\alpha+1)}\int_0^{s\wedge t} (s-u)^{\alpha-1}(t-u)^{\alpha-1}(s+t-2u)\Lambda_{i}(u)\dif u.
\end{equation}
\end{proposition}

\begin{proof}
In light of (\ref{integraleSkellamGeneralizzatoInTerminiDiPoisson}) we limit ourselves to the study of the moments of $N^\alpha$, the fractional integral of an arbitrary non-homogeneous Poisson process.
\begin{align*}
\mathbb{E}N^\alpha(t) &= \frac{1}{\Gamma(\alpha)}\int_0^t (t-s)^{\alpha-1} \Lambda(s)\dif s \\
%&=\frac{1}{\Gamma(\alpha)}\int_0^t (t-s)^{\alpha-1} \int_0^s \lambda(u)\dif u\dif s\\ 
&=\frac{1}{\Gamma(\alpha)}\int_0^t  \lambda(u)\dif u \int_u^t(t-s)^{\alpha-1} \dif s\\ 
& = \frac{1}{\Gamma(\alpha+1)}\int_0^t (t-u)^{\alpha}
 \lambda(u)\dif u.
\end{align*}

It is easy to see that the covariance (\ref{covarianzaIntegraleSkellamGeneralizzato}) reduces to the variance in (\ref{momentiPrimoSecondoIntegraleSkellamGeneralizzatoNonOmogeneo}) when $s=t$, so we limit ourselves to prove the following, for $0\le s\le t$,
\begin{align*}
 \text{Cov}&\Big(N^\alpha(s), N^\alpha(t)\Big)\\
& =  \frac{1}{\Gamma(\alpha)^2}\int_0^s (s-u)^{\alpha-1} \dif u \int_0^t(t-w)^{\alpha-1} \dif w\,  \text{Cov}\Big(N(u), N(w)\Big)\\
& = \frac{1}{\Gamma(\alpha)^2} \Biggl( \int_0^s (s-u)^{\alpha-1} \dif u \int_u^s(t-w)^{\alpha-1} \dif w \Lambda(u) \\
&\ \ \ + \int_0^s (t-w)^{\alpha-1} \dif w \int_w^s(s-u)^{\alpha-1} \dif u \Lambda(w) + \int_0^s (s-u)^{\alpha-1} \dif u \int_s^t(t-w)^{\alpha-1} \dif w \Lambda(u)\Biggr)\\
& =  \frac{1}{\Gamma(\alpha)\Gamma(\alpha+1)}\Biggl( \int_0^s (s-u)^{\alpha-1}(t-u)^{\alpha}\Lambda(u)\dif u +  \int_0^s (s-w)^{\alpha}(t-w)^{\alpha-1}\Lambda(w)\dif w \Biggr),
\end{align*}
where in the second step we suitably separated the integration set $[0,s]\times[0,t]$ and we considered the covariance of the Poisson process.
\end{proof}

Now we restrict ourselves to the case of the classical integral of a homogeneous Skellam process, i.e. we study $S^1$ with constant rates. In this case we have a representation in terms of a compound Poisson multiplied by the time variable. This result readily follows from Remark \ref{proposizioneSkellamGeneralizzatoPoissonComposto} and the following Proposition.

\begin{proposition}\label{proposizioneIntegraleProcessoPoissonComposto}
Let $Z$ be a compound Poisson process such that  $Z(t) = \sum_{k=1}^{N(t)}  X_k,\ t\ge0$, where $N$ is an independent Poisson processes of rate $\lambda>0$ and $X_1,\dots$ are i.i.d. random variables. Then, the Riemann integral $Y$ of $Z$ is
\begin{equation}\label{integraleProcessoPoissonComposto}
Y(t) = \int_0^t Z(s)\dif s \stackrel{d}{=} t\sum_{k=1}^{N(t)} X_kU_k
\end{equation}
where $U_1,U_2,\dots\sim Unif(0,1)$ are i.i.d. random variables, independent from the other terms.
\end{proposition}

\begin{proof}
It is well-known that the compound Poisson process is a L\'{e}vy process, therefore, the proposition easily follows from Lemma 1 of \cite{X2018} which provides the following general result on the moment generating function of the integral of a L\'{e}vy process. Let $X$ be a L\'{e}vy process, then for $\gamma\in\mathbb{R},\ t\ge0$,
\begin{equation}
\mathbb{E}e^{i\gamma \int_0^t X(s)\dif s} = \text{exp}\Biggl( t \int_0^1 \ln \mathbb{E} e^{i\gamma t z X(1)}\dif z \Biggr).
\end{equation}
Now, for the compound Poisson process $Z$, with $X,U$ being copies of $X_k$ and $U_k$ respectively, we derive
\begin{align*}
\mathbb{E}e^{i\gamma \int_0^t Z(s)\dif s}&=  \text{exp}\Biggl( t \int_0^1 \ln \mathbb{E} e^{i\gamma t z Z(1)}\dif z \Biggr) \\
&=  \text{exp}\Biggl( t \int_0^1\Big[-\lambda + \lambda \mathbb{E} e^{i\gamma  t z X}\Bigr] \dif z \Biggr)\\
%& =  \text{exp}\Biggl( -t \lambda + t\lambda \int_0^1\mathbb{E} e^{i\gamma t z  X}\dif z \Biggr)\\
& =  \text{exp}\Biggl( -t \lambda \Bigl[1- \mathbb{E} e^{i\gamma t UX} \Bigr]\Biggr)
\end{align*}
which coincides with the moment generating function of the right-hand side of (\ref{integraleProcessoPoissonComposto}).
\end{proof}

Hence, if $S\sim NHGSP (\lambda_i, i\in \mathcal{I} )$, then $S^1(t) =  t\sum_{k=1}^{N(t)} X_kU_k$ where the $X_k$ are given in (\ref{SkellamGeneralizzatoComePoissonComposto}). As usual, a compound Poisson representation can be extremely useful to obtain further properties of the process, as described after the proof of Remark \ref{proposizioneSkellamGeneralizzatoPoissonComposto}. In this case, it is also worthwhile to note the ease of deriving the moments. Indeed, for a compound Poisson $Z$ defined as in Proposition \ref{proposizioneIntegraleProcessoPoissonComposto}, for $s,t\ge0$, we have $\mathbb{E}Z(t) = \mathbb{E}N(t)\mathbb{E}X,\ \text{Var}Z(t) = \text{Var}N(t)\big(\mathbb{E}X\big)^2 + \mathbb{E}N(t)\text{Var}X$, $\text{Cov}\big(Z(s),Z(t)\big) = \mathbb{E}N(s\wedge t)\text{Var} X$, where the $X$ is a copy of the $X_k$.

%\begin{remark}[Running average]
Note that by means of Proposition \ref{proposizioneIntegraleProcessoPoissonComposto} the running average of a compound Poisson process $Z$, $Z_A (t)=\int_0^t Z(s)\dif s /t $, is pointwise a compound Poisson process with modified jumps as displayed in (\ref{integraleProcessoPoissonComposto}). 
%Thus, we can derive the iterated running average. We denote by $Z^{(m)}$ the $m$-th fold running average, $m\in\mathbb{N}$, and $Z_A^{(0)} = Z$. Now, with $M>0,\ t\ge0$, we have
%$$Z_A^{(M)}(t) = \int_0^t \frac{\dif t_1}{t_1} Z_A^{(M-1)}(t_1) =  \int_0^t \frac{\dif t_1}{t_1}\cdots \int_0^{t_{M-1}}\frac{\dif t_M}{t_M} Z(t_M) \stackrel{d}{=} \sum_{k=1}^{N(t)} X_kU_k^{(1)}\cdots U_k^{(M)},$$
%where in the last equality we used Proposition \ref{proposizioneIntegraleProcessoPoissonComposto} and $U_k^{(m)}\sim Unif(0,1)$ are i.i.d. random variables for $m=1\dots,M, \ k\in\mathbb{N}$. We point out that $\prod_{m=1}^M U_k^{(m)}$ converges in mean to $0$ as $M\rightarrow \infty$ (for every $k$). Therefore, if $|X_k|$ has finite mean, then $Z^{(M)}_A(t)\rightarrow 0$ in mean. Indeed, with $X, U^{(m)}$ being copies of $X_k$ and $U^{(m)}_k$ respectively, $\mathbb{E}|Z(t)| = \mathbb{E}N(t)\mathbb{E}|X|\mathbb{E}\prod_{m=1}^M U^{(m)} \rightarrow 0$.
%\hfill$\diamond$
%\end{remark}

\section{Fractional generalized Skellam processes}\label{sezioneSkellamFrazionario}

In this section we study fractional versions of the non-homogeneous generalized Skellam process. Inspired by the work \cite{OT2015}, our approach is based on the fractionalization of the difference operator in the right-hand side of equation (\ref{equazioneDifferenzeDifferenzialeSkellamGeneralizzato}), using Bernstein functions; we refer to \cite{SSV2012} for a detailed discussion. In Section \ref{sottosezioneSkellamFrazionarioNelTempo} we use also a time-fractional operator.

We recall that $f:[0,\infty)\longrightarrow [0,\infty)$ is a Bernstein function if $f\in C^\infty, \ (-1)^n \dif^n f/\dif x^n \le 0\ \forall \ n\ge1$ and can be expressed as
\begin{equation}
f(x) = a+bx+ \int_0^\infty\Bigl( 1-e^{-xw}\Bigr) \nu(\dif w), \ \ \ x\ge0,
\end{equation}
where $a,b\ge0$ and $\nu$ is a L\'{e}vy measure, i.e. such that $\int_0^\infty (s\wedge1)\nu(\dif s)<\infty$.
Bernstein functions are related to non-decreasing L\'{e}vy processes, also known as subordinators. Indeed, for each Bernstein function $f$ there exists a subordinator $\mathcal{H}_f$ such that $f$ is the L\'{e}vy symbol of $\mathcal{H}_f$, i.e. $\mathbb{E}e^{-\mu\mathcal{H}_f(t)} = e^{-tf(\mu)}, \ \mu, t\ge0$. Hereafter we assume $a=b=0$.

Now, for the sake of clarity we restrict ourselves to the case where $\mathcal{I}\subset \mathbb{Z}$, and we rewrite (\ref{equazioneDifferenzeDifferenzialeSkellamGeneralizzato}) as
\begin{equation}\label{equazioneDifferenzeDifferenzialeSkellamGeneralizzatoOperatori}
  \frac{\dif }{\dif t} p_n(t)= -\sum_{i\in\mathcal{I}} \lambda_i(t) \Bigl(I-B^i\Bigr) p_n(t), \ \ \ t \ge0, n\in \mathbb{Z},
\end{equation}
where $I$ is the identity operator and $B$ is the backward operator, such that $B^i p_n(t) = p_{n-i}(t)\ \forall \ t$ (meaning that $B^i = F^{-i}$ if $i<0$, with $F$ being the forward operator). 

We here state the following Theorem concerning the fractional version of (\ref{equazioneDifferenzeDifferenzialeSkellamGeneralizzatoOperatori}). 

\begin{theorem}\label{teoremaDefinzioniEquivalentiSkellamGeneralizzatoBernstein}
Let $\mathcal{I}\subset \mathbb{Z}\setminus\{0\}, \ |\mathcal{I}|<\infty$, integrable $\lambda_i : [0,\infty] \longrightarrow [0, \infty)$ and $f_i$ be a Bernstein function $\forall\ i\in\mathcal{I}$. Then, the solution to the fractional difference-differential problem 
\begin{equation}\label{equazioneDifferenzeDifferenzialeSkellamGeneralizzatoOperatoriBernstein}
  \frac{\dif }{\dif t} p_n(t)= -\sum_{i\in\mathcal{I}} f_i\Bigl( \lambda_i(t)\big(I-B^i\big) \Bigr)p_n(t), \ \ \ t \ge0,\ n\in \mathcal{S} =  \bigcup_{m=1}^\infty m\mathcal{I},\ \ \ p_n(0) =\begin{cases}\begin{array}{l l}1,& n=0,\\0,&n\not=0,\end{array}\end{cases} 
\end{equation}
is the probability law of a stochastic process $S_f$ with independent increments, $S_f(0) = 0\ a.s.$ and, for $t\ge0,\ k\in \mathcal{S}$,
\begin{equation}\label{definizioneInfinitesimaleSkellamGeneralizzatoNonOmogeneoBernstein}
P\{S_f(t+\dif t) = k+n\,|\,S_f(t) = k\} = 
\begin{cases}
\begin{array}{l l}
\displaystyle\sum_{\substack{m\in\mathbb{N}, i\in\mathcal{I}\\mi = n}} \frac{\lambda_i(t)^m}{m!}\dif t \int_0^\infty e^{-\lambda_i(t) w} w^m\nu_i(\dif w)+ o(\dif t), & n\in \bigcup_{k\ge1}^\infty k\mathcal{I},\\
1-\sum_{i\in\mathcal{I}} f_i\big(\lambda_i(t)\big) \dif t + o(\dif t), & n=0, \\
o(\dif t), & \text{otherwise}.
\end{array}
\end{cases}
\end{equation}
We define $S_f$ non-homogeneous generalized Bernstein-fractional Skellam process and for $u$ in the neighborhood of $0$,
\begin{equation}\label{funzioneGeneratriceProbabilitaSkellamFrazionarioBernstein}
\mathbb{E} u^{S_f(t)} = \exp\Biggl(-\sum_{i\in \mathcal{I}} \int_0^t f_i\Big(\lambda_i(s)\big(1-u^i\big)\Big)\dif s\Biggr), \ \ \ t\ge0.
\end{equation} 
\end{theorem}

Note that in \eqref{definizioneInfinitesimaleSkellamGeneralizzatoNonOmogeneoBernstein}, $k\mathcal{I} = \{ki\,:\, i\in\mathcal{I}\}$.
\\We denote the non-homogeneous generalized Bernstein-fractional Skellam process with $S_f\sim NHGBFSP\Bigl((f_i,\lambda_i),\ i\in\mathcal{I}\Bigr)$. We omit the letter "$N$" when we refer to the homogeneous case, that is when the rate functions $\lambda_i$ are all constants.

\begin{proof}
We begin by proving that the probability law of $S_f$, $p^f_n(t) = P\{S_f(t) = n\}$ satisfies equation (\ref{equazioneDifferenzeDifferenzialeSkellamGeneralizzatoOperatoriBernstein}).
First, we observe that for $a_{m,i}$ arbitrary positive numbers,
$$ \sum_{n\in\cup_{k\ge1}k\mathcal{I}}\  \sum_{\substack{m\in\mathbb{N}, i\in\mathcal{I}\\im = n}} a_{m, i}= \sum_{ i\in\mathcal{I}}\  \sum_{n\in\cup_{k\ge1}k\mathcal{I}} \ \sum_{\substack{m\in\mathbb{N}\\im = n}} a_{m, i} =  \sum_{ i\in\mathcal{I}} \sum_{k=1}^\infty  a_{k,i},$$
and therefore we obtain
\begin{align*}
\sum_{n\in\cup_{k\ge1}k\mathcal{I}}\, \sum_{\substack{m\in\mathbb{N}, i\in\mathcal{I}\\mi = n}}  \frac{\lambda_i(t)^m}{m!} \int_0^\infty e^{-\lambda(t) w} w^m\nu_i(\dif w)&= \sum_{i\in\mathcal{I}} \sum_{k=1}^\infty \frac{\lambda_i(t)^k}{k!} \int_0^\infty e^{-\lambda_i(t) w} w^k\nu_i(\dif w)\\
& = \sum_{i\in\mathcal{I}} \int_0^\infty \Big(e^{\lambda_i(t)w}-1\Big) e^{-\lambda_i(t) w} \nu_i(\dif w) \\
& =\sum_{i\in\mathcal{I}} f_i\Big(\lambda_i(t)\Big).
\end{align*}
This clarifies the expression for $n=0$ in \eqref{definizioneInfinitesimaleSkellamGeneralizzatoNonOmogeneoBernstein} and  implies that
\begin{equation}\label{SkellamBernsteinFormulaAiuto}
\sum_{k\in\cup_{h\ge1}h\mathcal{I}}\  \sum_{\substack{m\in\mathbb{N}, i\in\mathcal{I}\\im = k}} p_{n-im}(t) a_{m, i}=  \sum_{m=1}^\infty \sum_{ i\in\mathcal{I}} p_{n-im}(t) a_{m,i}.
\end{equation}

Now, keeping in mind (\ref{SkellamBernsteinFormulaAiuto}), from \eqref{definizioneInfinitesimaleSkellamGeneralizzatoNonOmogeneoBernstein}, by means of usual arguments we derive that $p_n^f(t)$ satisfies, for $t\ge0$ and $n\in\mathcal{S}$, the following (first) equality
\begin{align}
 \frac{\partial }{\partial t} p_n^f(t)&= -\sum_{i\in\mathcal{I}} f_i\Big(\lambda_i(t)\Big)p_n(t) + \sum_{i\in\mathcal{I}} \sum_{m=1}^\infty \frac{\lambda_i(t)^m}{m!} p_{n-im}(t) \int_0^\infty w^m e^{-\lambda_i(t)w}\nu_i(\dif w) \label{equazioneDefinizioneInfinitesimaSkellamFrazionarioBernstein}\\
& =  -\sum_{i\in\mathcal{I}} \int_0^\infty \Biggl[ p_n(t) -e^{-\lambda_i(t)w}\biggl( p_n(t) + \sum_{m=1}^\infty \frac{w^m\lambda_i(t)^m}{m!} p_{n-im}(t)\biggr) \Biggr]\nu_i(\dif w) \nonumber\\
& = -\sum_{i\in\mathcal{I}} \int_0^\infty \Biggl[ p_n(t) - e^{-\lambda_i(t)w} \sum_{m=0}^\infty  \frac{w^m\lambda_i(t)^m}{m!} B^{im} p_{n}(t)\Bigg] \nu_i(\dif w)\nonumber\\
& = -\sum_{i\in\mathcal{I}} \int_0^\infty \Biggl[ p_n(t) - e^{-\lambda_i(t)w\big(I-B^i\big)} p_{n}(t)\Bigg] \nu_i(\dif w)\nonumber\\
& = -\sum_{i\in\mathcal{I}} f_i\Bigl( \lambda_i(t)\big(I-B^i\big) \Bigr)p_n(t),\nonumber
\end{align}
which coincides with \eqref{equazioneDifferenzeDifferenzialeSkellamGeneralizzatoOperatoriBernstein}.

Now, from the equation (\ref{equazioneDefinizioneInfinitesimaSkellamFrazionarioBernstein}) we can obtain the generating function (\ref{funzioneGeneratriceProbabilitaSkellamFrazionarioBernstein}), by proceeding as shown in the proof of Theorem \ref{teoremaDefinzioniEquivalentiSkellamGeneralizzato}. Let $G^f_t(u) = \mathbb{E}u^{S_f(t)}$ with $u$ in a neighborhood of $0$, then equation (\ref{equazioneDefinizioneInfinitesimaSkellamFrazionarioBernstein}) turns into
\begin{align}
 \frac{\partial }{\partial t} G^f_t(u) &= \sum_{n=-\infty}^\infty u^n \frac{\partial }{\partial t} p_n^f(t)\nonumber\\
&= -\sum_{i\in\mathcal{I}}  f_i\Big(\lambda_i(t)\Big) G^f_t(u)+ \sum_{i\in\mathcal{I}} \sum_{m=1}^\infty \frac{\lambda_i(t)^m}{m!}u^{im} \int_0^\infty w^m e^{-\lambda_i(t)w}\nu_i(\dif w)G^f_t(u)\nonumber\\
&= -G^f_t(u)\sum_{i\in\mathcal{I}}\Bigg[ \int_0^\infty\Big(1-e^{-\lambda_i(t)w}\Big)\nu_i(\dif w) +  \int_0^\infty \Bigl(e^{\lambda_i(t)wu^i}-1\Big) e^{-\lambda_i(t)w}\nu_i(\dif w)\Bigg] \nonumber \\
& = -\sum_{i\in\mathcal{I}} f_i\Bigl( \lambda_i(t)\big(1-u^i\big) \Bigr)G^f_t(u),\nonumber
\end{align}
which, in light of the initial condition in (\ref{equazioneDifferenzeDifferenzialeSkellamGeneralizzatoOperatoriBernstein}) yields (\ref{funzioneGeneratriceProbabilitaSkellamFrazionarioBernstein}).
\end{proof}

\begin{example}
If $\mathcal{I} = \{1\}$, then $S_f$ reduces to the counting process discussed in \cite{OT2015}.
Let $K\in\mathbb{N}$. If $\mathcal{I} = \{1,\dots,K\}$ we have a fractional version of the Poisson process of order $K$, see \cite{BS2024, DcMM2016, SMU2020}. If $\mathcal{I} = \{-K,\dots, -1,1,\dots,K\}$ we have a fractional version of the Skellam process of order $K$, see \cite{GKL2020, KK2024}.
\hfill$\diamond$
\end{example}

By using the arguments in Remark \ref{remarkIncrementiSkellamGeneralizzato}, we obtain that the increments of $S_f$ have the following probability generating function, for $0\le s\le t$,
$$\mathbb{E}u^{S_f(s+t)-S_f(s)} = \text{exp}\Biggl(-\sum_{i\in \mathcal{I}} \int_0^t f_i\Big(\lambda_i(s+w)\big(1-u^i\big)\Big)\dif w\Biggr).$$
Hence, the increments (which are independent) behave as a fractional Skellam process themselves, $\big\{S_f(s+t)-S_f(s)\big\}_{t\ge0}\sim NHGBFSP\Bigl(\big(f_i,\lambda_i(s+\cdot)\big),\ i\in\mathcal{I}\Bigr)$.
If and only if the rate functions $\lambda_i$ are constant, the increments are stationary as well.

\begin{remark}
We point out that also in this fractional case we have some result of the type of Proposition \ref{proposizioneCombinazioneLineareProcessiSkellamGenerlizzati}. In particular point ($i$) holds true, meaning that for $a\in\mathbb{Z}$,  $aS_f \sim NHGBFSP\Biggl( a\mathcal{I}, \Big( \big( f_{i/a},\lambda_{i/a}\big),\ i\in a\mathcal{I}\Big) \Biggr)$.
\\
Concerning the summation we can state the following. Let $S^{(j)}_f\sim NHGBFSP\Bigl(\big( f_{i,j},\lambda_i\big),\ i\in\mathcal{I}\Bigr),$ with $j=1,\dots,J\in\mathbb{N}$, and $f_{i,j}$ being a Bernstein function for each $i,j$. Then,
\begin{equation}\label{sommaSkellamFrazionari}
 \sum_{j=1}^J S^{(j)}_f(t)   \sim NHGBFSP\Biggl(\bigg( \sum_{j=1}^J f_{i,j},\lambda_i\bigg),\ i\in\mathcal{I}\Biggr),\ \ \ t\ge0.
\end{equation}
Formula (\ref{sommaSkellamFrazionari}) is well posed because the sum of Bernstein functions is still Bernstein. To prove \eqref{sommaSkellamFrazionari} the interested reader can use the probability generating function \eqref{funzioneGeneratriceProbabilitaSkellamFrazionarioBernstein}.
\hfill$\diamond$
\end{remark}

\begin{remark}[First passage times]
If we assume $\mathcal{I}\subset\mathbb{N}$, the fractional process is non-decreasing and the results presented in Section \ref{tempiPrimoPassaggioSkellamGeneralizzato} can be extended. Indeed, by means of the same arguments one can obtain that, denoting with $T_n = \inf\{t\ge0\,:\,S(t)\ge n\}$ the first passage time through the level $n\in\mathbb{N}$, for $|u|\le 1$,
$$  \sum_{n=1}^\infty u^n P\{T_n > t \}= \frac{u}{1-u} \text{exp}\Biggl(-\sum_{i\in\mathcal{I}} \int_0^t f_i\Big( \lambda_i(s) \big(1-u^i\big)\Big) \dif s \Biggr). $$
Furthermore, by denoting with $q_n(t) = P\{T_n > t\}$, one can obtain that
\begin{equation*}
\frac{\dif }{\dif t}q_n(t) = -\sum_{i\in\mathcal{I}} f_i\Bigl( \lambda_i(t) \big(I-B^i\big)\Big) q_n(t) ,\ \ t\ge0,\ n\in\mathbb{N}, 
\end{equation*}
with $q_n(0) = P\{T_n>0\} = 1, \ n\ge1$, and $q_n(t)=0,\ t\ge0,\ n\le0$.
Finally, also formula (\ref{funzioneGeneratriceMomentiTempoPrimoPassaggio}) still holds and if the rate functions are constant one obtains
$$ \sum_{n=1}^\infty u^n \mathbb{E} T^r_n  = \frac{u\,\Gamma(r+1)}{1-u}\Biggl(\sum_{i\in\mathcal{I}}f_i\Bigl( \lambda_i\big(1-u^i\big)\Big)\Biggr)^{-r},\ \ \ r>0.$$
\hfill$\diamond$
\end{remark}

\subsection{Homogeneous case}

The homogeneous case is particularly interesting since the Bernstein-fractional process can be represented as the linear combination of time-changed Poisson processes.

\begin{proposition}\label{proposizioneScomposizioneSkellamBernsteinFrazionarioInProcessiPoisson}
Let $S_f\sim HGBFSP\Bigl((f_i,\lambda_i),\ i\in\mathcal{I}\Bigr)$, then
\begin{equation}\label{SkellamFrazionarioBernsteinGeneralizzatoInTerminiDiPoisson}
S_f(t)=\sum_{i\in\mathcal{I}} i N_i\big(H_{f_i}(t)\big),\ \ \ t\ge0,
\end{equation}
where, for $ i\in\mathcal{I}$, $N_i$ are independent Poisson processes with rate functions $\lambda_i$ and $H_{f_i}$ are independent subordinators with L\'{e}vy symbol $f_i$.
\end{proposition}

\begin{proof}
To prove (\ref{SkellamFrazionarioBernsteinGeneralizzatoInTerminiDiPoisson}) it is sufficient to show the following relationship, for $t\ge0$ and $u$ in a neighborhood of $0$,
\begin{align*}
\mathbb{E}u^{\sum_{i\in\mathcal{I}} i N_i\big(H_{f_i}(t)\big)}& = \prod_{i\in\mathcal{I}}\mathbb{E}\Biggl[  \mathbb{E}\Big[u^{ i N_i\big(H_{f_i}(t)\big)}\Big| H_{f_i},i\in\mathcal{I}  \Big]\Bigg]\\
& = \prod_{i\in\mathcal{I}}\mathbb{E}\Biggl[  e^{-\lambda_i H_{f_i}(t)\big(1-u^i\big)}  \Bigg]\\
& = \prod_{i\in\mathcal{I}} e^{-tf_i\bigl(\lambda_i(1-u^i)\bigr)}
\end{align*}
which coincides with (\ref{funzioneGeneratriceProbabilitaSkellamFrazionarioBernstein}) when $\lambda_i$ are constants.
\end{proof}

%We recall that  $H_{f_1}+ H_{f_2} \stackrel{d}{=}H_{f_1+f_2}$, with $f_1+f_2$ still being a Bernstein function. Therefore, by combining Proposition \ref{proposizioneCombinazioneLineareProcessiSkellamGenerlizzati} and Proposition \ref{proposizioneScomposizioneSkellamBernsteinFrazionarioInProcessiPoisson} one obtain the following. Let $S^{(j)}_f\sim HGBFSP\Biggl(\Big( f_{i,j},\lambda_i\Big),\ i\in\mathcal{I}\Biggr),$ with $j=1,\dots,J\in\mathbb{N}$, and $f_{i,j}$ being a Bernstein function for each $i,j$. Then, for $t\ge0$,
%$$\sum_{j=1}^J S^{(j)}_f(t) =  \sum_{i\in\mathcal{I}} i \sum_{j=1}^J  N^{(j)}_i\big(H_{f_{i,j}}(t)\big) = \sum_{i\in\mathcal{I}} i N_i\Biggl(\sum_{j=1}^J H_{f_{i,j}}(t) \Biggr)  \sim HGBFSP\Biggl(\bigg( \sum_{j=1}^J f_{i,j},\lambda_i\bigg),\ i\in\mathcal{I}\Biggr). $$
%The interested reader may also consider the case in which $\mathcal{I}$ depends on $j$.
%\\

Note that, in the case of non-homogeneous rate functions, the probability generating function of the right--hand of (\ref{SkellamFrazionarioBernsteinGeneralizzatoInTerminiDiPoisson}) reads
$$\mathbb{E}u^{\sum_{i\in\mathcal{I}} i N_i\big(H_{f_i}(t)\big)} =\prod_{i\in\mathcal{I}}\mathbb{E}\Biggl[  e^{-\big(1-u^i\big)\int_0^{H_{f_i}(t)}\lambda_i(s) \dif s}  \Bigg]$$
and the time-changed formulation is not holding.

In light of Proposition \ref{proposizioneScomposizioneSkellamBernsteinFrazionarioInProcessiPoisson} we can obtain the moments of the homogeneous Bernstein-fraction process, with $0\le s\le t$,
\begin{align}
&\mathbb{E}S_f(t) = \sum_{i\in\mathcal{I}} i\lambda_i \mathbb{E}H_{f_i}(t),\ \ \ \text{Var}S_f(t) = \sum_{i\in\mathcal{I}} i^2 \Big(\lambda_i^2 \text{Var}H_{f_i}(t) + \lambda_i \mathbb{E}H_{f_i}(t) \Bigr) \label{momentiSkellamBernsteinOmogeneo}\\
&\text{Cov}\Bigl(S_f(s), S_f(t)\Big)  =  \sum_{i\in\mathcal{I}} i^2 \text{Var}N_i\big(H_{f_i}(s)\big) =  \text{Var}S_f(s) ,  \nonumber
\end{align}
where we used the fact that $H_{f_i}(s)\le H_{f_i}(t)\ a.s\ \forall\ i$. 
\\

The interested reader can refer to the papers \cite{BS2024, GKL2020, KK2024, OT2015} for the study of some particular cases of the Poisson or Skellam processes (of order $K$) time-changed with Bernstein subordinators. Note that equation \eqref{sommaSkellamOmogeneoConTempoModificato} can be of interesting when the class of the subordinators is closed with respect the sum (like for the gamma subordinator).
\\
%Finally, we observe that by keeping in mind that a sum of compound Poisson processes is a compound Poisson process, by means of Theorem/Proposition ?? of \cite{OT2015} one can express the homogeneous fractional generalized Skellam as the limit of compound Poisson processes. 

Inspired by Theorem 2.2 of \cite{OT2015} we express the homogeneous fractional generalized Skellam in terms of (the limit of) a compound Poisson process. 

First, we recall the following useful lemma which states that the class of the compound Poisson processes is closed with respect to finite linear combinations.

\begin{lemma}\label{lemmaSommaProcessioPoissonComposto}
Let $ \mathcal{I} = \{1,\dots,n\}$ with $n\in\mathbb{N}$, $\{a_i\}_{i\in\mathcal{I}}$ be a collection of real numbers and $Z_i$ be a compound Poisson process such that  $Z_i(t) = \sum_{k=1}^{N_i(t)}  X^{(i)}_k,\ t\ge0$, where $N_i$ is an independent Poisson process of rate $\lambda_i>0$ and $X_1^{(i)},X_2^{(i)},\dots$ are i.i.d. random variables, for $\forall\ i\in\mathcal{I}$. Then $\sum_{i\in\mathcal{I}}a_i Z_i$ is a compound Poisson process such that
\begin{equation}\label{combinazioneLineareProcessiPoissonComposto}
\sum_{i\in\mathcal{I}}a_i Z_i (t)\stackrel{d}{=} \sum_{k=1}^{N_\mathcal{I}(t)} X_k^\mathcal{I} ,
\end{equation}
where $N_\mathcal{I} = \sum_{i\in\mathcal{I}}N_i$ and $X_k^\mathcal{I} = \sum_{i\in\mathcal{I}} a_i X_k^{(i)} \mathds{1}(B_k^{(i)} = 1)$, with $B_k\sim Multinomial\bigl(\lambda_i/\sum_{i\in\mathcal{I}}\lambda_i\bigr)$ i.i.d. $\forall\ k$.
\end{lemma}

Note that in (\ref{combinazioneLineareProcessiPoissonComposto}) the jumps $X_k^{\mathcal{I}}$ are mixtures of the original ones with weights given by the rates of the Poisson processes.

For the sake of completeness, the interested reader can find the proof of Lemma \ref{lemmaSommaProcessioPoissonComposto} in Appendix \ref{appendiceDimostrazioneLemmaSommaPoissonComposti}.

\begin{proposition}
Let $u_i(n) = \int_0^\infty P\{N_i(t)\ge n\}\nu_i(\dif t), \ n\in\mathbb{N}, \ i \in\mathcal{I}$ with $N_i$ independent homogeneous Poisson processes with rate $\lambda_i>0$ and 
\begin{equation}\label{PoissonCompostoPerLimiteASkellamFrazionarioBernstein}
S_n(t) = \sum_{k=1}^{N\big(t\sum_{i\in\mathcal{I} }u_i(n)\big)} X_{k,n},
\end{equation}
where $N$  is a Poisson process of rate $1$ and for $k,n\in\mathbb{N}, \ X_{k,n}$ is the following mixture
\begin {equation*}
X_{n,k} = \sum_{i\in\mathcal{I}} iX_{k,n}^{(i)}\mathds{1}\big(B_{k,n}^{(i)} = 1\big)\ \text{ with }\ P\big\{X_{k,n}^{(i)} = m\big\} = \frac{1}{u_i(n)}\int_0^\infty P\{N_i(t)=m\}\nu_i(\dif t),
\end{equation*}
for $m\in\mathbb{N}$ and $B_{k,n}\sim Multinomial\big(u_i(n)/\sum_{i\in\mathcal{I}}u_i(n)\big)$. Then, if $S_f\sim HGBFSP (\lambda_i, i\in \mathcal{I} )$, 
$$ S_n(t) \stackrel{d}{\substack{\xrightarrow{\hspace*{1.4cm}}\\ n\longrightarrow0} }S_f(t),\ \ \ t\ge0.$$
In addition, if $\int_0^\infty \nu_i(\dif w)<\infty \ \forall\ i$, then $S_0 (t) = S_f(t),\ t\ge0$.
\end{proposition}

\begin{proof}
In view of Theorem 2.2 of \cite{OT2015} we have that, for 
\begin{equation}\label{formulaPerApprosimazioneSkellamBernstein}
	Z_n^{(i)}(t) = \sum_{k=1}^{N_i\big(tu_i(n)\big)} X_{k,n}^{(i)} \stackrel{d}{\substack{\xrightarrow{\hspace*{1.4cm}}\\ n\longrightarrow0} }N_i\big(H_{f_i}(t)\big), \ t\ge0, \ i\in\mathcal{I}.
	\end{equation}
Hence, by keeping in mind the representation in terms of time-changed Poisson processes of the generalized Bernstein-fractional Skellam process, \eqref{SkellamFrazionarioBernsteinGeneralizzatoInTerminiDiPoisson}, we have that $\sum_{i\in\mathcal{I}}iZ_n^{(i)}(t)  \stackrel{d}{\rightarrow}S_f(t)$ as $n\longrightarrow0$. Finally, by means of Lemma \ref{lemmaSommaProcessioPoissonComposto}, $\sum_{i\in\mathcal{I}}iZ_n^{(i)}(t) $ reduces to the compound Poisson in (\ref{PoissonCompostoPerLimiteASkellamFrazionarioBernstein}).

We point that in the case of finite L\'{e}vy measures $\nu_i, \ i\in\mathcal{I}$, then $u_i(0)$ exists finite and therefore relationship \eqref{formulaPerApprosimazioneSkellamBernstein} modifies into $Z_n^{(i)}(t) = \sum_{k=1}^{N_i(t u_i(0))} X_{k,n}^{(i)}$ and the final equality in the statement follows by repeating the used argument.
\end{proof}

\subsubsection{Time-fractional derivative}\label{sottosezioneSkellamFrazionarioNelTempo}

We now consider the homogeneous case in which the time derivative of equation \eqref{equazioneDifferenzeDifferenzialeSkellamGeneralizzatoOperatoriBernstein} is replaced by the Caputo-Dzerbashyan fractional derivative of order $\alpha >0$.

\begin{theorem}\label{teoremaSkellamBernsteinCaputo}
Let $\alpha\in(0,1)$, $\mathcal{I}\subset \mathbb{Z}\setminus\{0\}, \ |\mathcal{I}|<\infty$, $\lambda_i >0$ and $f_i$ be a Bernstein function for $i\in\mathcal{I}$. Then, the solution to the fractional difference-differential problem 
\begin{equation}\label{equazioneDifferenzeDifferenzialeSkellamGeneralizzatoOperatoriBernsteinCaputo}
  \frac{\partial^\alpha }{\partial t^\alpha} p_n(t)= -\sum_{i\in\mathcal{I}} f_i\Bigl( \lambda_i\big(I-B^i\big) \Bigr)p_n(t), \ \ \ t \ge0,\ n\in  \bigcup_{m=1}^\infty m\mathcal{I},\ \ \ p_n(0) =\begin{cases}\begin{array}{l l}1,& n=0,\\0,&n\not=0,\end{array}\end{cases} 
\end{equation}
is the probability law of the process $S_{f,\alpha} = S_f\circ L_\alpha$ where $S_f\sim HGBFSP\Bigl((f_i,\lambda_i),\ i\in\mathcal{I}\Bigr)$ and $L_\alpha$ is an independent inverse of the subordinator of order $\alpha$.
\end{theorem}

Note that the above result could be easily generalized to the case of the inverse of a Bernstein subordinator by means of the convolutional-type derivative (for instance, by suitably applying Theorem 4.1 of \cite{CO2026} or the arguments for Theorem 1 in \cite{BS2024}); we refer to \cite{K2011, T2015} for some details on this operator.

\begin{proof}
Keeping in mind that the rate functions are constant, following the line of the proof of Theorem \ref{teoremaDefinzioniEquivalentiSkellamGeneralizzatoBernstein}, from (\ref{equazioneDifferenzeDifferenzialeSkellamGeneralizzatoOperatoriBernsteinCaputo}) we obtain, for $t\ge0$ and $|u|\le1$,
\begin{align}
 \frac{\partial^\alpha }{\partial t^\alpha}G_t(u) = -\sum_{i\in\mathcal{I}} f_i\Bigl( \lambda_i\big(1-u^i\big) \Bigr)G_t(u).
\end{align}
where $G_t(u) = \sum_{n=0} u^n p_n(t)$. By means of the Laplace transform one can show that the solution to (\ref{equazioneDifferenzeDifferenzialeSkellamGeneralizzatoOperatoriBernsteinCaputo}), with initial condition $G_0(u) =1$, is 
\begin{equation}\label{funzioneGeneratriceProbabilitaSkellamBernsteinCaputo}
G_t(u) = E_{\alpha, 1}\Biggl(-t^\alpha\sum_{i\in\mathcal{I}}f_i\Bigl(\lambda_i\big(1-u^i\big) \Big)\Biggr).
\end{equation}
Finally, we show that the probability generating function of $S_f\big(L_\alpha(t)\big)$ coincides with (\ref{funzioneGeneratriceProbabilitaSkellamBernsteinCaputo}) for all $t\ge0$.
\begin{align}
\mathbb{E}u^{S_f\big(L_\alpha(t)\big)}& = \mathbb{E}\Bigg[\mathbb{E}\Big[u^{S_f\big(L_\alpha(t)\big)}\Big|L_{\alpha}(t)\Big]\Bigg]  \nonumber\\
& = \mathbb{E}\ \text{exp}\Biggl(- L_{\alpha}(t)\sum_{i\in \mathcal{I}}  f_i\Big(\lambda_i\big(1-u^i\big)\Big)\Biggr)     \label{passaggioBernsteinCaputoUsoFunzioneGeneratriceBernstein}\\
& =  E_{\alpha, 1}\Biggl(-t^\alpha\sum_{i\in\mathcal{I}}f_i\Bigl(\lambda_i\big(1-u^i\big) \Big)\Biggr) \nonumber
\end{align}
where in (\ref{passaggioBernsteinCaputoUsoFunzioneGeneratriceBernstein}) we used \eqref{funzioneGeneratriceProbabilitaSkellamFrazionarioBernstein} and in the last step we used the Laplace transform of the law of $L_\alpha(t)$, i.e. $\mathbb{E}e^{-\mu L\alpha(t)} = E_{\alpha, 1}\big(-t^\alpha \mu\big), \ \mu,t\ge0$. 
\end{proof}

\begin{remark}[Pseudo-processes]
We point out that Theorem \ref{teoremaSkellamBernsteinCaputo} can be extended to the case of $\alpha>1$. This implies that $L_\alpha$ is an independent pseudo-inverse of the pseudo-subordinator of order $\alpha$ and, therefore, $S_{f,\alpha}$ is not a genuine stochastic process, but a pseudo-process, meaning that it has a real pseudo-measure. We refer to \cite{CO2025} and references therein for the details on the formalization of pseudo-processes, pseudo-subordinators and their inverses. 
\hfill$\diamond$
\end{remark}

From the composition in Theorem \ref{teoremaSkellamBernsteinCaputo}, the interested reader can obtain the moments of $S_{f,\alpha}$ by means of the formulas in (\ref{momentiSkellamBernsteinOmogeneo}).

%%%%%%%%%%%%%%%%%%%%%%%%%%

\subsection*{\large{Declarations}}

\textbf{Ethical Approval.} This declaration is not applicable.
 \\
\textbf{Competing interests.}  The authors have no competing interests to declare.
\\
\textbf{Authors' contributions.} Both authors equally contributed in the preparation and the writing of the paper.
 \\
\textbf{Funding.} The authors received no funding.
 \\
\textbf{Availability of data and materials.} This declaration is not applicable.

%\subsection*{Acknowledgments}

% ---------------------------------------------------------------------------------------------------------


 \appendix

\section{Generating function related to first passage times}\label{appendiceDimostrazioneFunzioneGeneratriceTempiPrimoPassaggio}

%For the sake of completeness 
We prove formula (\ref{generatriceTempiPrimoPassaggioRipartizione}). The reader can equivalently derive (\ref{generatriceTempiPrimoPassaggio}).
For $n\in\mathbb{N},\ t\ge0$ and $|u|<1$,
\begin{align*}
\sum_{n=1}^\infty u^n P\{T_n \le t\} &= \sum_{n=1}^\infty u^n  \sum_{k=n}^\infty P\{S(t) = k\}\\
& = \sum_{k=1}^\infty  P\{S(t) = k\} \frac{u^{k+1}-u}{u-1} \\
& = \frac{u}{u-1}\Biggl[\sum_{k=1}^\infty P\{S(t) = k\} \big(u^k - 1\big) + \big(u^0-1\big)P\{S(t) = 0\}\Biggr] \\
&  = \frac{u}{u-1}\sum_{k=0}^\infty P\{S(t) = k\} (u^k-1)\\
& =\frac{u}{u-1}\Bigl(G_t(u) - 1\Bigr).
\end{align*}

\section{Sum of compound Poisson processes}\label{appendiceDimostrazioneLemmaSommaPoissonComposti}

We prove Lemma \ref{lemmaSommaProcessioPoissonComposto}. Let $t\ge0$ and $u$ in the neighborhood of $0$. By assuming that $X^{(i)}$, $X^\mathcal{I}$ and $B$ are copies of the $X_k^{(i)},\ X^\mathcal{I}_k$ and $B_k$ respectively, $\forall\ k, i\in\mathcal{I}$ we arrive at the following probability generating function,
\begin{align*}
\mathbb{E}u^{X^{\mathcal{I}}} &= \mathbb{E}\Biggl[\mathbb{E}\Bigl[ u^{\sum_{i\in\mathcal{I}} a_iX^{(i)}\mathds{1}(B^{(i)} =1)}\Big| B \Bigr]\Biggr]\\
& = \sum_{j\in\mathcal{I}} P\{B^{(j)} = 1\}\mathbb{E}\Biggl[\mathbb{E}\Bigl[ u^{\sum_{i\in\mathcal{I}} a_iX^{(i)}\mathds{1}(B^{(i)} =1)}\Big| B^{(j)} = 1 \Bigr]\Biggr] \\
& = \sum_{j\in\mathcal{I}} \frac{\lambda_j}{\sum_{i\in\mathcal{I}}\lambda_i} \mathbb{E}u^{a_jX^{(j)}}.
 \end{align*}
Thus, 
\begin{align*}
\mathbb{E}u^{\sum_{k=1}^{N_\mathcal{I}(t)} X^{\mathcal{I}}_k} & = \text{exp}\Biggl(-\sum_{i\in\mathcal{I}}\lambda_i\Bigl[ 1-\mathbb{E}u ^{X^\mathcal{I}}\Big]\Biggr)\\
& = \text{exp}\Biggl(-\sum_{i\in\mathcal{I}}\lambda_i\Biggl[ 1-\sum_{j\in\mathcal{I}} \frac{\lambda_j}{\sum_{i\in\mathcal{I}}\lambda_i} \mathbb{E}u^{a_jX^{(j)}}\Bigg]\Biggr)\\
& =  \text{exp}\Biggl(-\sum_{i\in\mathcal{I}}\lambda_i\Biggl[ 1- \mathbb{E}u^{a_jX^{(j)}}\Bigg]\Biggr)\\
& = \mathbb{E}u^{\sum_{i\in \mathcal{I}} a_iZ_i(t)}.
\end{align*}


%%       \cite{<label>} ==> Jones et al. [21]
%%       \citep{<label>} ==> [21]
%%


\footnotesize{

}

\end{document}